\documentclass[11pt]{article}

\usepackage{comment}
\usepackage{subcaption}
\usepackage{amsmath,amssymb,amsthm}
\usepackage{amsbsy}
\usepackage{enumerate,verbatim}
\usepackage{graphicx}
\usepackage{diagbox}
\usepackage{mathdots}
\usepackage{longtable}
\usepackage{enumitem}
\usepackage{algorithm,float}
\usepackage{algorithmic}
\usepackage{caption}
\usepackage[a4paper,left=36mm,right=36mm,top=15mm,bottom=15mm]{geometry}
\usepackage{bm}
\usepackage{latexsym}
\usepackage{amsfonts}
\usepackage{authblk}       
\usepackage[numbers,sort&compress]{natbib}
\usepackage{multirow}
\usepackage{titlesec}
\usepackage{indentfirst}
\usepackage[colorlinks=true,linkcolor=blue,anchorcolor=blue,citecolor=blue]{hyperref}
\usepackage{setspace}
\usepackage{rotating}
\usepackage{multirow}
\usepackage{booktabs}
\usepackage{siunitx}
\usepackage{cleveref}
\newtheorem{theorem}{Theorem}[section]
\newtheorem{lemma}{Lemma}[section]
\newtheorem{example}{Example}[section]
\newtheorem{problem}{Problem}[section]

\makeatletter
\newtheorem{definition}{Definition}

{\par\vspace{\abovedisplayskip}\noindent
	\textbf{{\small Funding}}\hspace{0.5em}\ignorespaces}
{\par\vspace{\belowdisplayskip}}
\makeatother

\newenvironment{ethics}
{\par\vspace{\abovedisplayskip}\noindent
	\textbf{{\small Ethics approval}}\hspace{0.5em}\ignorespaces}
{\par\vspace{\belowdisplayskip}}
\makeatother

\newenvironment{consent}
{\par\vspace{\abovedisplayskip}\noindent
	\textbf{{\small Consent for publication}}\hspace{0.5em}\ignorespaces}
{\par\vspace{\belowdisplayskip}}
\makeatother

\newenvironment{materials}
{\par\vspace{\abovedisplayskip}\noindent
	\textbf{{\small Materials availability}}\hspace{0.5em}\ignorespaces}
{\par\vspace{\belowdisplayskip}}
\makeatother

\newenvironment{code}
{\par\vspace{\abovedisplayskip}\noindent
	\textbf{{\small Code availability}}\hspace{0.5em}\ignorespaces}
{\par\vspace{\belowdisplayskip}}
\makeatother

\newenvironment{data}
{\par\vspace{\abovedisplayskip}\noindent
	\textbf{{\small Data Availability}}\hspace{0.5em}\ignorespaces}
{\par\vspace{\belowdisplayskip}}
\makeatother

\newenvironment{competing}
{\par\vspace{\abovedisplayskip}\noindent
	\textbf{{\small Competing interests}}\hspace{0.5em}\ignorespaces}
{\par\vspace{\belowdisplayskip}}
\makeatother

\newenvironment{authorcon}
{\par\vspace{\abovedisplayskip}\noindent
	\textbf{{\small Author contribution}}\hspace{0.5em}\ignorespaces}
{\par\vspace{\belowdisplayskip}}
\makeatother

\newenvironment{constforpar}
{\par\vspace{\abovedisplayskip}\noindent
	\textbf{{\small Consent to participate}}\hspace{0.5em}\ignorespaces}
{\par\vspace{\belowdisplayskip}}
\makeatother

\numberwithin{equation}{section}
\titleformat{\section}{\fontsize{14pt}{16.8pt}\selectfont\bfseries}{\thesection}{1em}{}
\makeatletter
\newenvironment{breakablealgorithm}
{
		\begin{center}
			\refstepcounter{algorithm}
			\hrule height.8pt depth0pt \kern2pt
			\renewcommand{\caption}[2][\relax]{
				{\raggedright\textbf{\ALG@name~\thealgorithm} ##2\par}%
				\ifx\relax##1\relax 
				\addcontentsline{loa}{algorithm}{\protect\numberline{\thealgorithm}##2}%
				\else 
				\addcontentsline{loa}{algorithm}{\protect\numberline{\thealgorithm}##1}%
				\fi
				\kern2pt\hrule\kern2pt
			}
		}{
		\kern2pt\hrule\relax
	\end{center}
}
\makeatother

\makeatletter
\newcommand{\Rmnum}[1]{\expandafter\@slowromancap\romannumeral #1@}
\makeatother

\makeatletter
\def\ps@pprintTitle{}
\makeatother

\title{A Cayley-free Two-Step Algorithm  for Inverse  Singular Value Problems}
\author{ Jiechang Fan\thanks{School of Mathematical Sciences, Zhejiang Normal University, Jinhua 321004, P. R. China
 (1637886951@qq.com).}\hspace{2em}
 Weiping Shen\thanks{Corresponding author. School of Mathematical Sciences, Zhejiang Normal University,
 Jinhua 321004, P. R. China (shenweiping@zjnu.cn). This author's work was supported in part by the National Natural Science Foundation of China (grant 12071441)}\hspace{2em}
Yusong Luo\thanks{School of Mathematical Sciences, Zhejiang Normal University, Jinhua 321004,  P. R. China (luoyusong@zjnu.edu.cn).}
\hspace{2em}
Enping Lou\thanks{College of Physics and Electronic Information Engineering, Zhejiang Normal University, Jinhua 321004,  P. R. China (louenping@zjnu.cn).}

}

\date{}

\begin{document}
	\maketitle
	\begin{abstract}
		\noindent  In this paper, we investigate numerical solutions for inverse singular value problems (for short, ISVPs) arising in various applications.   Inspired by the methodologies employed for inverse eigenvalue problems,
we propose a Cayley-free two-step algorithm for solving the ISVP.  Compared to the existing two-step algorithms for the ISVP, our algorithm eliminates the need for Cayley transformations and consequently avoids solving $2(m+n)$  linear systems during the computation of approximate singular vectors at each outer iteration.  Under the assumption that the  Jacobian matrix
	at a solution is nonsingular, we present a convergence analysis for the proposed algorithm and prove a cubic root-convergence rate. Numerical experiments are conducted to validate the effectiveness of our algorithm.
		
		\vspace{1em} 
		
		\vspace{1em} 
	\end{abstract}

	\section{Introduction}
	\noindent	The inverse singular value problems, which concerns  the reconstruction of matrices from prescribed singular value spectra,  have  received more and more attentions  due to its wide applications such as constructing the Toeplitz-related matrices with specified singular values, the passivity enforcement in nonlinear circuit simulation,
 the determination of mass distributions, orbital mechanics, irrigation theory, computed tomography  and so on.
 One may refer to \cite{isvpyy1,isvpyy2,isvpyy3,isvpyy4,isvpyy5,wjh2023} and the references therein for diverse applications of the ISVP. Let $m$, $n$ be positive integers and  let  $ \left\{A_i\right\}_{i=0}^n$ be a sequence of real $m\times n$ matrices.  Let $\bm{c}:=(c_1,c_2,\dots ,c_n)^T\in \mathbb{R}^n$ and define
	\begin{equation}\label{A}
		A(\bm{c}):=A_0+\sum_{i=1}^{n}c_iA_i.
	\end{equation}
	Let $ \left\{\sigma_i(\bm{c})\right\}_{i=1}^n $ be   the singular values of the matrix  $ A(\bm{c}) $ ordering with
	$	\sigma_1(\bm{c}) \geq \sigma_2(\bm{c}) \geq \cdots \geq \sigma_n(\bm{c}) \geq  0.$	
	Particularly, Chu originally proposed in \cite{newton dist} an important class of the ISVPs which  is described as follows. 	
	\begin{problem}[\textbf{ISVP}]\label{problem}
		
		Given $n$ nonnegative real numbers $\sigma_1^* \geq \sigma_2^* \geq \cdots \geq \sigma_n^*,$
		find a vector $\bm {c}^*\in \mathbb{R}^n$ such that
		\begin{equation}\label{iep}
			\sigma_i(\bm{c}^*)=\sigma_i^*,\quad 1\leq i \leq n. 	
		\end{equation}
	\end{problem}
	Our interest in the present paper is the topic of numerically solving the ISVP \eqref{iep} by iterative algorithm.  Let $\bm{f}: \mathbb{R}^n\rightarrow \mathbb{R}^n$ be defined by
	\begin{equation}\label{f}
		\bm{f}(\bm{c}):=(\sigma_1(\bm{c})-\sigma_1^*,\;\sigma_2(\bm{c})-\sigma_2^*,\;\dots,\;\sigma_n(\bm{c})-\sigma_n^*)^T.
	\end{equation} Then, as noted in \cite{ulm dist, shen ip 2020}, in the case when the given singular values $\{\sigma_{i}^{\ast}\}_{i=1}^n$ are distinct and positive, i.e.,
	\begin{equation}\label{dispos}
		\sigma_{1}^{\ast}>\sigma_{2}^{\ast}>\cdots>\sigma_{n}^{\ast}>0,
	\end{equation}
	the function $\bm{f}$ is analytic and the Jacobian matrix $\boldsymbol{f}'(\bm{c})$ of $\boldsymbol{f}$ at $\bm{c}\in\mathbb{R}^{n}$ is given by
	\begin{equation}\label{jacobi}
		[\boldsymbol{f}'(\bm{c})]_{ij} := \bm{u}_i(\bm{c})^T  A_j \bm{v}_i(\bm{c}),\quad 1 \leq i,\;j\leq n,
	\end{equation}
	where $U(\mathbf{c}):=[{\bf u}_1 ({\bf c}),\;\ldots$ ${\bf u}_m({\bf c})]$ and $V(\mathbf{c}):=[{\bf v}_1 ({\bf c}),\;\ldots,\;{\bf v}_n({\bf c})]$
	are two orthogonal matrices of the left   and right singular vectors of $A(\mathbf{c})$, respectively.
	Clearly,    when the   singular values $\{\sigma_{i}^{\ast}\}_{i=1}^n$  satisfy  \eqref{dispos},
	the classical Newton  method can be applied  (to the equation $\bm{f}(\bm{c})=0$)
	and exhibits local quadratic convergence toward a solution of  \eqref{iep}  under  the nonsingularity assumption of  the Jacobian matrix at a
	solution.  However, as  pointed out by \cite{shenAppl. Numer. Math.2016},
	the  Newton  method requires a complete
	singular value decomposition   at each outer iteration which reduces the algorithm's practical efficiency  when dealing with large-scale problems.

	Alternatively, solving \eqref{iep} is equivalent to finding a point in the intersection of the manifolds $\mathcal{A}$ and $\mathcal{M}(\Sigma^*)$, where
	$\Sigma^*:={\rm diag}(\sigma^{*}_{1},\;\ldots,\;\sigma^{*}_{n})\in\mathbb{R}^{m\times n}$, $\mathcal{A}:=\{A(\mathbf{c})\;|\;\mathbf{c}\in\mathbb{R}^n\}$ and
	$\mathcal{M}(\Sigma^*):=\{U\Sigma^*V^T\;|\;U,V\text{are orthogonal matrices}\}$.
	Based on this equivalence and making use of  the Cayley transform,  Chu designed in \cite{newton dist} a Newton-type method for solving the ISVP \eqref{iep} which computes approximate singular vectors instead of exact singular vectors  at each iteration.  However, the Newton-type method in \cite{newton dist} requires solving a Jacobian equation in each outer iteration, which can be costly for large-scale problems. To reduce computational cost, an inexact version of the Newton-type method was proposed in \cite{inexact newton} which solves the Jacobian equations approximately using a suitable stopping criterion. An alternative cost-reduction strategy is the Ulm-like method for the ISVP \eqref{iep} proposed by Vong et al. \cite{ulm dist} which   avoids solving  the (approximate) Jacobian equations. Recently, inspired by Aishima's  work  \cite{ashima iep}  for the inverse eigenvalue problems,
	Wei and Chen  \cite{ashima isvp}  developed an iterative  algorithm for the ISVP  \eqref{iep} which, compared with the (inexact) Newton-type methods and the Ulm-like method,   refined orthogonality without the Cayley transform in
	obtaining approximate singular vectors.
	
	Note that all the algorithms mentioned above are quadratically
	convergent. In order to speed up the convergence rate, several two-step algorithms have been developed for solving the ISVP   \eqref{iep}. Particularly,
	by executing the inexact Newton-type procedure  with shared Jacobian matrices, Ma and Chen
	proposed in \cite{2 inexact newton} a two-step inexact Newton-type  method
	which is at least super quadratically convergent in  root sense provided that the given singular values $\{\sigma_{i}^{\ast}\}_{i=1}^n$  are distinct and positive.
	Recently, by applying the two-step Ulm-Chebyshev iterative procedure (cf. \cite{ulmtype}) to the equation $\bm{f}(\bm{c})=0$,
	\cite{2 ulm dist} designed a two-step Ulm-Chebyshev-like method for solving the ISVP \eqref{iep}. As  described  in \cite{2 ulm dist}, the two-step Ulm-Chebyshev-like method eliminates the need to solve the (approximate) Jacobian equations and  achieves higher convergence rate‌: at least cubic convergence in root sense.
	It should be mentioned that the above two-step algorithms  obtain approximate singular vectors by $4$ Cayley transforms at each outer iteration.
	As remarked in \cite{ashima iep}, however, each Cayley transform requires solving $n$ (or $m$) linear systems and takes $O(n^3)$ (or $O(m^3)$) arithmetic operations to produce an orthogonal matrix from a skew-symmetric matrix.

	The purpose of the present paper is to propose a Cayley-free two-step algorithm for solving the ISVP   \eqref{iep}.
	Under the assumption that the  Jacobian matrix
	at a solution $\bm {c}^*$ is nonsingular, we prove that the proposed algorithm is at least cubic convergence in root sense.
	It should be noted that, compared with the existing two-step algorithms for solving the ISVP \eqref{iep}, the proposed algorithm requires no  Cayley transforms (and so $2(m+n)$ linear systems) in obtaining approximate singular vectors at each outer iteration and therefore results in higher computational efficiency.

	The paper is organized as follows. In Section \ref{section2}, we propose a  Cayley-free two-step algorithm for solving the ISVP\eqref{iep}.  The convergence analysis of the proposed algorithm is given in Section \ref{section3}. Finally, to illustrate the efficiency of our algorithm, we report some numerical tests in Section \ref{section4}.
	
	\section{The Proposed Algorithm}\label{section2}
	\noindent We begin this section with  some standard notation.
	As usual, let $\mathbb{N}$ be the set of all nonnegative integers,  $\mathbb{R}^{n}$
	and $\mathbb{R}^{m\times n}$ be the $n$-dimensional Euclidean space and the set of all real $m\times n$ matrices. Let $ \Vert \cdot \Vert $  denote the Euclidean vector norm on $\mathbb{R}^{n}$ and its corresponding induced matrix norm on $\mathbb{R}^{m\times n}$,  while $ \Vert \cdot \Vert_F $ specifically denotes the Frobenius norm on
	$\mathbb{R}^{m\times n}$.
	$I$ stands for the identity matrix in $\mathbb{R}^{n\times n}$ and $\text{diag}(a_1,a_2,\dots,a_n)\in\mathbb{R}^{m\times n}$  is the diagonal matrix with $\{a_i\}_{i=1}^n\subset\mathbb{R}$ being its  diagonal elements. For a matrix $ A\in\mathbb{R}^{m\times n} $ and vector $\bm{b} \in \mathbb{R}^n$, we denote the
	$(i,j)$-th entry of $ A $ by $ [A]_{ij} $ and the $i$-th component of $\bm{b}$ by $[\bm{b}]_{i}$.
	$ A^{-1} $ denotes the inverse matrix of $ A $.
	For a given $\bm{c}\in \mathbb{R}^n$, we use
	$B(\bm{c},r)$ to denote the open ball at $\bm{c}$ with radius $r > 0$. Let $ \left\{\sigma_i(\bm{c})\right\}_{i=1}^n $ be the singular values of the matrix  $ A(\bm{c}) $  with the order
	$	\sigma_1(\bm{c}) \geq \sigma_2(\bm{c}) \geq \cdots \geq \sigma_n(\bm{c}) \geq  0 $, 	$\{\bm{u}_i(\bm{c})\}_{i=1}^m $ and $\{\bm{v}_i(\bm{c})\}_{i=1}^n $
	be the corresponding normalized left singular vectors  and normalized right singular vectors, respectively.
	We also  assume throughout this paper  that  the given singular values $ \left\{ \sigma_i^*\right\}_{i=1}^n $ are all positive and distinct(cf. \cite{newton dist,new dist con,inexact newton,ulm dist,ashima isvp,2 inexact newton,2 ulm dist}), i.e.,
	\begin{equation}
\sigma_1^*>\sigma_2^*>\cdots>\sigma_n^*>0.\label{distinct}
	\end{equation}	
	For simplicity, we write $\boldsymbol{\sigma}^*:=(\sigma_1^*,\sigma_2^*,\cdots,\sigma_n^*)^T\in \mathbb{R}^{n}$,
	$\Sigma^{*}:= \operatorname{diag}(\sigma_{1}^{*}, \ldots, \sigma_{n}^{*}) \in \mathbb{R}^{m \times n}$, and define
	\begin{align*}
		&\mathcal{I}_1 := \left\{ (i, j) \mid 1 \leq i ,\; j \leq n,\; i \neq j \right\},\quad	\mathcal{I}_2 := \left\{ (i, j) \mid n+1 \leq i \leq m,\;1 \leq j \leq n   \right\},\\
		&\mathcal{I}_3 := \left\{ (i, j) \mid   1 \leq i \leq n,\;n+1 \leq j \leq m  \right\},\quad\mathcal{I}_4 := \left\{ (i, j) \mid n+1 \leq i,\; j \leq m,\; i \neq j  \right\}.
	\end{align*}
	Moreover, we assume throughout this paper that $\bm{c}^*$ is a solution of the ISVP \eqref{iep}.
We recall the two-step Ulm-Chebyshev-like method proposed by Ma \cite{2 ulm dist} for solving the ISVP \eqref{iep}.
	\begin{breakablealgorithm}\label{algorithm Ulm}
		\caption{The two-step Ulm-Chebyshev-like method}
		\begin{algorithmic}[1]
			
			\STATE Given $\bm{c}^0\in \mathbb{R}^{n}$, compute the singular values \(\{\sigma_i(\bm{c}^0)\}_{i=1}^n\), the orthogonal left singular vectors $\{\bm{u}_i(\bm{c}^0)\}_{i=1}^m$ and right singular vectors \(\{\bm{v}_i(\bm{c}^0)\}_{i=1}^n\) of \(A(\bm{c}^0)\). Let
			\begin{equation}
				U_0 := [\bm{u}_1^0, \bm{u}_2^0, \ldots, \bm{u}_m^0] = [\bm{u}_1(\bm{c}^0), \bm{u}_2(\bm{c}^0), \ldots, \bm{u}_m(\bm{c}^0)],\nonumber
			\end{equation}	
			\begin{equation}
				V_0 := [\bm{v}_1^0, \bm{v}_2^0, \ldots, \bm{v}_n^0] = [\bm{v}_1(\bm{c}^0), \bm{v}_2(\bm{c}^0), \ldots, \bm{v}_n(\bm{c}^0)],\nonumber
			\end{equation}	
			\begin{equation}
				\boldsymbol{\sigma}_0 := (\sigma_1^0, \sigma_2^0, \ldots, \sigma_n^0)^T = (\sigma_1(\bm{c}^0), \sigma_2(\bm{c}^0), \ldots, \sigma_n(\bm{c}^0))^T.\nonumber
			\end{equation}	
			Form the Jacobian matrix \( J_0\) and the vector \( \bm{b}^0 \) by
			\[
			[J_0]_{ij} := ({\bm{u}_i^0})^T A_j \bm{v}_i^0, \quad
			[\bm{b}^0]_i := ({\bm{u}_i^0})^T A_0 \bm{v}_i^0, \quad
			1 \leq i,\;j \leq n.
			\]
			
			Set \( B_0 := J_0^{-1} \) and \( \bm{s}^0 := (s_1^0, s_2^0, \ldots, s_n^0)^T = \boldsymbol{\sigma}^* \).
			\STATE For $k=0,1,\cdots $, until convergence, do.
			
			(1) Compute $\bm{y}^k$ by
			\begin{equation}
				\bm{y}^k = \bm{c}^k - B_k \left( J_k \bm{c}^k + \bm{b}^k - \boldsymbol{\sigma}^* \right).\nonumber
			\end{equation}
			
			(2) Compute the matrix $A(\bm{y}^k)$.
			
			(3) Compute the matrix $D_k:= U_k^TA(\bm{y}^k)V_k$.
			
			(4) Compute the skew-symmetric matrices $ X_k \in \mathbb{R}^{m \times m}$ and $ Y_k \in \mathbb{R}^{n \times n} $ by		
			\begin{equation}
				[X_k]_{ij} = 0, \quad (i,j) \in \mathcal{I}_4,	\nonumber
			\end{equation}
			\begin{equation}
				[X_k]_{ij} = -[X_k]_{ji} = \frac{[D_k]_{ij}}{s_j^k}, \quad (i,j) \in \mathcal{I}_2,	\nonumber	
			\end{equation}
			\begin{equation}
				[X_k]_{ij} = -[X_k]_{ji} = \frac{s_i^k [D_k]_{ji} + s_j^k [D_k]_{ij}}{(s_j^k)^2 - (s_i^k)^2}, \quad (i,j) \in \mathcal{I}_1,	\nonumber
			\end{equation}
			\begin{equation}	
				[Y_k]_{ij} = -[Y_k]_{ji} = \frac{s_i^k [D_k]_{ij} + s_j^k [D_k]_{ji}}{(s_j^k)^2 - (s_i^k)^2}, \quad (i,j) \in \mathcal{I}_1.	\nonumber
			\end{equation} 		
			
			(5) Compute \( {Z}_k := [\bm{z}_1^k, \bm{z}_2^k, \ldots, \bm{z}_m^k] \) and \(N_k := [\bm{n}_1^k, \bm{n}_2^k, \ldots, \bm{n}_n^k] \) by solving
			\begin{align}
				\left( I + \frac{1}{2} X_k \right) {Z}_k^T = \left( I - \frac{1}{2}X_k \right) U_k^T,\label{ma1'}\\
				\left( I + \frac{1}{2} Y_k \right) {N}_k^T = \left( I - \frac{1}{2} Y_k \right) V_k^T.\label{ma1}
			\end{align}

			(6) Compute \(\overline{\boldsymbol{\sigma}}(\bm{y}^k) := (\overline{\sigma}_1(\bm{y}^k), \overline{\sigma}_2(\bm{y}^k), \ldots, \overline{\sigma}_n(\bm{y}^k))^T\) by
			\[
			\overline{\sigma}_i(\bm{y}^k) = ({\bm{z}_i^k})^T A(\bm{y}^k) \bm{n}_i^k, \quad 1 \leq i \leq n.
			\]
			
			(7) Compute \( \bm{c}^{k+1} \) by
			\[
			\bm{c}^{k+1} = \bm{y}^k - B_k (\overline{\boldsymbol{\sigma}}(\bm{y}^k) - \boldsymbol{\sigma}^*).
			\]
			
			(8) Set \(\overline{\bm{s}}^k := (\overline{s}_1^k, \overline{s}_2^k, \ldots, \overline{s}_n^k)^T = \boldsymbol\sigma^* + \overline{\bm{t}}^k\), where
			\[
			\overline{\bm{t}}^k := \left( \overline{t}_1^k, \overline{t}_2^k, \ldots, \overline{t}_n^k \right)^T = (I - J_k B_k)(\overline{\boldsymbol\sigma}(\bm{y}^k)- \boldsymbol{\sigma^*}).
			\]
			
			(9) Compute the matrix $ A(\bm{c}^{k+1}) $.
			
			(10) Compute the matrix $ \overline{D}_k := U_k^T A(\bm{c}^{k+1}) V_k - U_k^T A(\bm{y}^k) V_k + Z_k^T A(\bm{y}^k) N_k. $
			
			(11) Compute the skew-symmetric matrices $ \overline{X}_k \in \mathbb{R}^{m \times m}  $ and $ \overline{Y}_k \in \mathbb{R}^{n \times n}$ by
			\[		[\overline{X}_k]_{ij} = 0, \quad (i,j) \in \mathcal{I}_4,	\]	
			\[	[\overline{X}_k]_{ij} = -[\overline{X}_k]_{ji} = \frac{[\overline{D}_k]_{ij}}{\overline{s}_j^k}, \quad (i,j) \in \mathcal{I}_2,		\]		
			\[		[\overline{X}_k]_{ij} = -[\overline{X}_k]_{ji} = \frac{\overline{s}_i^k [\overline{D}_k]_{ji} + \overline{s}_j^k [\overline{D}_k]_{ij}}{(\overline{s}_j^k)^2 - (\overline{s}_i^k)^2}, \quad (i,j) \in \mathcal{I}_1,		\]
			\[		[\overline{Y}_k]_{ij} = -[\overline{Y}_k]_{ji} = \frac{\overline{s}_i^k [\overline{D}_k]_{ij} + \overline{s}_j^k [\overline{D}_k]_{ji}}{(\overline{s}_j^k)^2 - (\overline{s}_i^k)^2}, \quad (i,j) \in \mathcal{I}_1.	\]
			
			(12) Compute \( U_{k+1}:= [\bm{u}_1^{k+1}, \bm{u}_2^{k+1}, \ldots, \bm{u}_m^{k+1}] \) and \( V_{k+1}:= [\bm{v}_1^{k+1}, \bm{v}_2^{k+1}, \ldots, \bm{v}_n^{k+1}] \) by solving
			\begin{align}
				\left( I + \frac{1}{2} \overline{X}_k \right) U_{k+1}^T = \left( I - \frac{1}{2} \overline{X}_k \right) Z_k^T,\label{ma2'}\\
				\left( I + \frac{1}{2} \overline{Y}_k \right) V_{k+1}^T = \left( I - \frac{1}{2} \overline{Y}_k \right) N_k^T.	\label{ma2}
			\end{align}

			(13) Compute \( \boldsymbol\sigma^{k+1}: = (\sigma_1^{k+1}, \sigma_2^{k+1}, \ldots, \sigma_n^{k+1})^T \) by
			\[\sigma_i^{k+1} = ({\bm{u}_i^{k+1}})^T A(\bm{c}^{k+1}) \bm{v}_i^{k+1}, \quad 1 \leq i \leq n.	\]
			
			(14) Form the approximate Jacobian matrix \( J_{k+1} \) and the vector \( \bm{b}^{k+1} \) by
			\[[J_{k+1}]_{ij} := ({\bm{u}_i^{k+1}})^T A_j \bm{v}_i^{k+1}, \quad 1 \leq i,\;j \leq n,	\]
			\[[\bm{b}]_{i}^{k+1} := ({\bm{u}_i^{k+1}})^T A_0 \bm{v}_i^{k+1}, \quad 1 \leq i \leq n.\]

			(15) Compute the Chebyshev matrices \( B_{k+1} \) by
			\[B_{k+1} = B_k + B_k (2I - J_{k+1} B_k)(I - J_{k+1} B_k).	\]
			
			(16) Set  \(\bm{s}^{k+1}:= (s_1^{k+1}, s_2^{k+1}, \ldots, s_n^{k+1})^T = \boldsymbol\sigma^* + \bm{t}^{k+1} \), where
			
			\[\bm{t}^{k+1}: = (t_1^{k+1}, t_2^{k+1}, \ldots, t_n^{k+1})^T = (I - J_{k+1} B_{k+1})(\boldsymbol\sigma^{k+1} - \boldsymbol\sigma^*).\]	
		\end{algorithmic}
	\end{breakablealgorithm}
	
	It should be noted that \autoref{algorithm Ulm} requires exact solution of $2(m+n)$ linear systems \eqref{ma1'}--\eqref{ma2}  in each iteration to generate the orthogonal matrices $Z_k$, $N_k$, $U_{k+1}$ and $V_{k+1}$, which are widely known as the Cayley transform (cf. \cite{Cayley1846}).
	As mentioned in	the introduction, each	Cayley transform takes $O(n^3)$ (or $O(m^3)$) arithmetic operations to produce an orthogonal	matrix from a skew-symmetric matrix.
	
	Inspired by \cite{ashima iep} and \cite{2 ulm dist}, we propose here  a Cayley-free two-step algorithm for solving the ISVP \eqref{iep}.
	For this end, we first note that $\{\bm{u}_i(\bm{c}^*)\}_{i=1}^n $ and $\{\bm{v}_i(\bm{c}^*)\}_{i=1}^n $
	are uniquely determined (ignore the signs) and that
	$\{\bm{u}_{i}(\bm{c}^*)\}_{i=n+1}^m$ are  not unique. Write $V(\bm{c}^*):=[\bm{v}_1(\bm{c}^*),\;\bm{v}_2(\bm{c}^*),\;\ldots,\;\bm{v}_n(\bm{c}^*)]$.
	Let $\mathcal{U}$ be the set of orthogonal matrices $U(\bm{c}^*)$ composed by the normalized left singular vectors $\{\bm{u}_{i}(\bm{c}^*)\}_{i=1}^m$ of $A(\bm{c}^*)$, i.e.,
	\begin{equation}
\mathcal{U}:=\{U(\bm{c}^*):=[\bm{u}_1(\bm{c}^*),\;\bm{u}_2(\bm{c}^*),\;\ldots,\;\bm{u}_m(\bm{c}^*)]\mid U(\bm{c}^*)^TA(\bm{c}^*)V(\bm{c}^*)=\Sigma^{*}\}.\nonumber
\end{equation}
	To proceed,  fix $U(\bm{c}^*)\in\mathcal{U}$ and $k\in\mathbb{N}$. Suppose that $\bm{c}^k$,  $U_k:=[\bm{u}_1^{k}, \bm{u}_2^{k}, \ldots, \bm{u}_m^{k}]$ and $V_k:=[\bm{v}_1^{k}, \bm{v}_2^{k}, \ldots, \bm{v}_n^{k}]$ are current computed approximations of  $\bm{c}^*$, $U(\bm{c}^*)$ and $V(\bm{c}^*)$, respectively.  Define
	\begin{equation}\label{minu^*}
		U_k^* := \arg\min_{U \in \mathcal{U}} \| U - U_k \|_F.
	\end{equation}
	Then,  $U_k^*:=[\bm{u}_1(\bm{c}^*),\;\bm{u}_2(\bm{c}^*),\;\ldots,\;\bm{u}_n(\bm{c}^*),\;\bm{u}_{n+1}^k(\bm{c}^*),\;\ldots,\;\bm{u}_{m}^k(\bm{c}^*)]$ is orthogonal and  ${U_k^*}^T A(\bm{c}^{*})V(\bm{c}^*)=\Sigma^{*} $.
	Let $X_k\in \mathbb{R}^{m \times m}$, $Y_k\in \mathbb{R}^{n \times n}$ such that
	\begin{equation}\label{defuv}
		U_k=U_k^*(I+X_k)\quad{\rm and}\quad V_k=V(\bm{c}^*)(I+Y_k).
	\end{equation}
	Thus, one checks
	\begin{equation}\label{u1}
		U_k^TU_k=I+X_k+X_k^T+X_k^TX_k,
	\end{equation}
	\begin{equation}\label{v1}
		V_k^TV_k=I+Y_k+Y_k^T+Y_k^TY_k
	\end{equation}
	and
	\begin{equation}\label{c^*1}
		U_{k}^{T} A(\bm{c}^{*}) V_{k} = \Sigma^{*} + \Sigma^{*} Y_{k} + X_{k}^{T} \Sigma^{*} + X_{k}^{T} \Sigma^{*} Y_{k}.
	\end{equation}
	Omitting the quadratic terms in \eqref{u1}--\eqref{c^*1}, we may expect that	$\overline{X}_k$, $\overline{Y}_k$ and $\overline{\bm{c}}^k$, the approximations of  $X_k$, $Y_k$ and $\bm{c}^*$,  satisfy the following three equations:
	\begin{equation}\label{ox1}
		U_k^TU_k=I+\overline{X}_k+\overline{X}_k^T,
	\end{equation}
	\begin{equation}\label{oy1}
		V_k^TV_k=I+\overline{Y}_k+\overline{Y}_k^T
	\end{equation}
	and
	\begin{equation}\label{oc1}
		[U_{k}^{T} A(\overline{\bm{c}}^k) V_{k}]_{ij} = [\Sigma^{*} + \Sigma^{*} \overline{Y}_k+\overline{X}_k^{T} \Sigma^{*}]_{ij}, \quad  (i,j)\in \mathcal{I}_1\cup\mathcal{I}_2.
	\end{equation}
	Considering the diagonal elements of \eqref{ox1} and \eqref{oy1}, we have
	\begin{equation}\label{diag ox}
		[\overline{X}_k]_{ii} = \frac{({{\bm{u}}_i^k})^T {\bm{u}}_i^k - 1}{2}, \quad 1 \leq i \leq m
	\end{equation}
	and
	\begin{equation}\label{diag oy}
		[\overline{Y}_k]_{ii} = \frac{({{\bm{v}}_i^k})^T {\bm{v}}_i^k - 1}{2}, \quad 1 \leq i \leq n.
	\end{equation}
	On the other hand, we note by \eqref{ox1} that  $$	U_k^TU_k\Sigma^{*}=\Sigma^{*}+\overline{X}_k\Sigma^{*}+\overline{X}_k^T\Sigma^{*},$$ which together with \eqref{oc1}
	gives
	\begin{equation}\label{oxy1}
		[U_{k}^{T}U_{k}\Sigma^{*} - U_{k}^{T}A(\overline{\bm{c}}^{k})V_{k}]_{ij} = [\overline{X}_{k}\Sigma^{*} - \Sigma^{*}\overline{Y}_{k}]_{ij}, \quad  (i,j)\in \mathcal{I}_1\cup\mathcal{I}_2.
	\end{equation}
	Similarly, thanks to \eqref{oy1} and \eqref{oc1}, we get
	\begin{equation}\label{oxy2}
		[\Sigma^{*}V_{k}^{T}V_{k} - U_{k}^{T}A(\overline{\bm{c}}^{k})V_{k}]_{ij} = [\Sigma^{*}\overline{Y}_{k}^T - \overline{X}_{k}^T\Sigma^{*}]_{ij}, \quad  (i,j)\in \mathcal{I}_1\cup\mathcal{I}_2.
	\end{equation}
	Writing $W_k:=U_{k}^{T}A(\overline{\bm{c}}^{k})V_{k}$, we deduce from \eqref{oxy1} and \eqref{oxy2} that
	\begin{equation}\label{upox}
		[\overline{X}_k]_{ij}=\frac{\sigma_i^*[W_k]_{ji}+\sigma_j^*[W_k]_{ij}-(\sigma_j^*)^2({{\bm{u}_i^k}})^T\bm{u}_j^k-\sigma_i^*\sigma_j^*({{\bm{v}_i^k}})^T\bm{v}_j^k}{ (\sigma_i^*)^2 - (\sigma_j^*)^2 },\;(i,j) \in \mathcal{I}_1
	\end{equation}
	and
	\begin{equation}\label{upoy}
		[\overline{Y}_k]_{ij} = \frac{\sigma_i^* [W_k]_{ij} +\sigma_j^* [W_k]_{ji}-\sigma_i^*\sigma_j^*({\bm{u}_i^k})^T\bm{u}_j^k-(\sigma_j^*)^2({\bm{v}_j^k})^T\bm{v}_i^k}{(\sigma_i^*)^2 - (\sigma_j^*)^2},\;(i,j) \in \mathcal{I}_1.
	\end{equation}
	Moreover, using \eqref{ox1} and \eqref{oxy2} again, we have
	\begin{equation}\nonumber
		({\bm{u}_i^k})^T\bm{u}_j^k	=[\overline{X}_{k}]_{ij}+[\overline{X}_{k}]_{ji},\quad(i,j) \in \mathcal{I}_2\\
	\end{equation}
	and
	\begin{equation}\nonumber
		-[W_k]_{ij}=-\sigma_j^{*}[\overline{X}_{k}]_{ji},\quad(i,j) \in \mathcal{I}_2.
	\end{equation}
	This gives
	\begin{equation}\label{oxother2}
		[\overline{X}_k]_{ij}=({\bm{u}_i^k})^T\bm{u}_j^k-\frac{[W_k]_{ij}}{\sigma_j^*}, \quad  (i,j)\in \mathcal{I}_2
	\end{equation}
and
	\begin{equation}\label{oxother3}
		[\overline{X}_k]_{ij}=\frac{[W_k]_{ji}}{\sigma_i^*}, \quad  (i,j)\in \mathcal{I}_3.
	\end{equation}
	Finally, in view of  \eqref{ox1}, we set
	\begin{equation}\label{oxfree}
		[\overline{X}_k]_{ji} =[\overline{X}_k]_{ij}= \frac{({\bm{u}_i^k})^T\bm{u}_j^k}{2}, \quad (i,j) \in \mathcal{I}_4.
	\end{equation}
	After that, motivated by \cite{ashima iep}, we compute the matrix $\overline{U}_k:=[\overline{\bm{u}}_1^{k}, \overline{\bm{u}}_2^{k}, \ldots, \overline{\bm{u}}_m^{k}]$ and $\overline{V}_k:=[\overline{\bm{v}}_1^{k}, \overline{\bm{v}}_2^{k}, \ldots, \overline{\bm{v}}_n^{k}]$ by
	\begin{equation}\label{defouv}	
		\overline{U}_k := U_k( I - \overline{X}_k)\quad{\rm and}\quad\overline{V}_k := V_k( I - \overline{Y}_k),
	\end{equation}
	where $ I - {X_k}$ and $  I - {Y_k}$ is the first-order approximation of $(I + {X_k})^{-1}$ and $(I + {Y_k})^{-1}$ obtained by the Neumann series.
	
	Next, we let $E_k \in \mathbb{R}^{m \times m}$ and $F_k \in \mathbb{R}^{n \times n}$ such that
	\begin{equation}\label{defzn}
		\overline{U}_k=\overline{U}_k^*(I+E_k)\quad{\rm and}\quad\overline{V}_k=V(\bm{c}^*)(I+F_k),
	\end{equation}
	where $\overline{U}_k^*$ satisfies that
	\begin{equation}\label{minou^*}
		\overline{U}_k^* := \arg\min_{U \in \mathcal{U}} \| U - \overline{U}_k \|_F.
	\end{equation}
	Similar to \eqref{u1}--\eqref{c^*1}, one has
	\begin{equation}\label{z1}
		{\overline{U}_k}^T\overline{U}_k=I+E_k+E_k^T+E_k^TE_k,
	\end{equation}
	\begin{equation}\label{n1}
		{\overline{V}_k}^T\overline{V}_k=I+F_k+F_k^T+F_k^TF_k
	\end{equation}
	and
	\begin{equation}\label{c^*2}
		{\overline{U}_k}^T A(\bm{c}^{*})\overline{V}_k = \Sigma^{*} + \Sigma^{*} F_{k} + E_{k}^{T} \Sigma^{*} + E_{k}^{T} \Sigma^{*}F_{k} .
	\end{equation}
	Omitting the quadratic terms in \eqref{z1}--\eqref{c^*2}, we  expect that	$\overline{E}_k$, $\overline{F}_k$ and $\bm{c}^{k+1}$, the approximations of  $E_k$, $F_k$ and $\bm{c}^*$,  satisfy the following three equations:
	\begin{equation}\label{zz}
		{\overline{U}_k}^T\overline{U}_k=I+	{\overline{E}}_k^T+	\overline{E}_k,
	\end{equation}
	\begin{equation}\label{nn}
		{\overline{V}_k}^T\overline{V}_k=I+	{\overline{F}}_k^T+	\overline{F}_k
	\end{equation}
	and
	\begin{equation}\label{ck+1}
		{\overline{U}_k}^T A(\bm{c}^{k+1})\overline{V}_k = \Sigma^{*} + \Sigma^{*}{\overline{F}_k} +{\overline{E}_k}^T \Sigma^{*}.
	\end{equation}
	Then, as we did for $\overline{X}_k$ and $\overline{Y}_k$,   we obtain by \eqref{zz}--\eqref{ck+1} that
	\begin{equation}\label{oe}
		[\overline{E}_k]_{ii}=\frac{({\overline{\bm{u}}_i^k})^T\overline{\bm{u}}_i^k-1}{2}, \quad 1 \leq i \leq m,
	\end{equation}
	\begin{equation}\label{upoe}
		[\overline{E}_k]_{ij}=
		\frac{\sigma_i^*[{\overline{W}_k}]_{ji}+\sigma_j^*[{\overline{W}_k}]_{ij}-(\sigma_j^*)^2({\overline{\bm{u}}_i^k})^T\overline{\bm{u}}_j^k-\sigma_i^*\sigma_j^*({\overline{\bm{v}}_i^k})^T\overline{\bm{v}}_j^k}{ (\sigma_i^*)^2 - (\sigma_j^*)^2 }, \; (i,j)\in \mathcal{I}_1,
	\end{equation}
	\begin{equation}\label{oeother}
		[\overline{E}_k]_{ij}=
		({\overline{\bm{u}}_i^k})^T\overline{\bm{u}}_j^k-\frac{[W_k]_{ij}}{\sigma_j^*}, \quad (i,j)\in \mathcal{I}_2,
	\end{equation}
	\begin{equation}\label{oeother3}
		[{\overline{E}}_k]_{ij}=\frac{[\overline{W}_k]_{ji}}{\sigma_i^*}, \quad (i,j)\in \mathcal{I}_3,
	\end{equation}
	\begin{equation}\label{oefree}
		[\overline{E}_k]_{ij}=\frac{({\overline{\bm{u}}_i^k})^T\overline{\bm{u}}_j^k}{2}, \quad (i,j)\in \mathcal{I}_4,
	\end{equation}
	\begin{equation}\label{of}
		[{\overline{F}_k}]_{ii}=\frac{({\overline{\bm{v}}_i^k})^T\overline{\bm{v}}_i^k-1}{2}, \quad 1 \leq i\leq n,
	\end{equation}
	\begin{equation}\label{upof}
		[{\overline{F}_k}]_{ij}=
		\frac{\sigma_i^* [W_k]_{ij} +\sigma_j^* [W_k]_{ji}-\sigma_i^*\sigma_j^*({\overline{\bm{u}}_i^k})^T\overline{\bm{u}}_j^k-(\sigma_j^*)^2({\overline{\bm{v}}_j^k})^T\overline{\bm{v}}_i^k}{(\sigma_i^*)^2 - (\sigma_j^*)^2},\;  (i,j)\in \mathcal{I}_1,
	\end{equation}
	where $\overline{W}_k:=\overline{U}_{k}^{T}A({\bm c}^{k+1})\overline{V}_{k}$.
	Finally, we form $U_{k+1}$ and $V_{k+1}$ by
	\begin{equation}\label{uvk+1}
		U_{k+1} := \overline{U}_k( I - \overline{E}_k)\quad {\rm and}\quad  V_{k+1} := \overline{V}_k( I - \overline{F}_k).
	\end{equation}
	
	We now  determine $\bm {c}^{k+1}$. To do this,  consider the nonlinear operator $\boldsymbol{g}: \mathbb{R}^{n}\rightarrow \mathbb{R}^{n}$ and recall  the two-step Ulm-Chebyshev method for solving  $ \boldsymbol{g}(\bm{x})=0$:
	\begin{equation}
		\left\{
		\begin{aligned}
			&	\bm{y}^k = \bm{x}^k - B_k \boldsymbol{g}(\bm{x}^k), \\
			&	\bm{x}^{k+1} = \bm{y}^k - B_k \boldsymbol{g}(\bm{y}^k), \label{ulmchebyshev} \\
			&	B_{k+1} = B_k + B_k (2I - \boldsymbol{g}'(\bm{x}^{k+1}) B_k) (I - \boldsymbol{g}'(\bm{x}^{k+1}) B_k),
		\end{aligned}
		\right.
	\end{equation}
	where $\bm{x}^0\in\mathbb{R}^{n}$ and $B_0\in \mathbb{R}^{n \times n}$ are given.
	To apply \eqref{ulmchebyshev} for solving the nonlinear system $ \boldsymbol{f}(\bm{c})=0$, we note by \eqref{f}  that, for each $1 \leq i \leq n$,
	\begin{equation}\begin{array}{lll}\label{fdeform} [\boldsymbol{f}(\bm{c}^k)]_i&=\bm{u}_i(\bm{c}^k)^TA(\bm{c}^k)\bm{v}_i(\bm{c}^k)-\sigma_i^*\\
			&=\bm{u}_i(\bm{c}^k)^TA(\bm{c}^k)\bm{v}_i(\bm{c}^k)-{\frac{\sigma_i^*}{2}(\bm{u}_i(\bm{c}^k)^T\bm{u}_i(\bm{c}^k)+{\bm{v}_i(\bm{c}^k)^T}\bm{v}_i(\bm{c}^k))}.
	\end{array}\end{equation}
	Moreover, we recall from  \eqref{jacobi} that
	\begin{equation}\label{jacobik}
		[\boldsymbol{f}'(\bm{c}^k)]_{ij} := {\bm{u}_i(\bm{c}^k)}^T  A_j \bm{v}_i(\bm{c}^k),\quad 1 \leq i,\;j\leq n.
	\end{equation}
	Then, thanks to the definition of $\boldsymbol{f}$ in \eqref{f}, one checks
	\begin{equation}\label{feq}
		\boldsymbol{f}(\bm{c}^k)=\boldsymbol{f}'(\bm{c}^k)\bm{c}^k+\bm{b}(\bm{c}^k),
	\end{equation}
	where
	$\bm{b}(\bm{c}^k)$ is defined by $$[\bm{b}(\bm{c}^k)]_i := {\bm{u}_i(\bm{c}^k)}^T A_0\bm{v}_i(\bm{c}^k)-{\frac{\sigma_i^*({\bm{u}_i(\bm{c}^k)}^T\bm{u}_i(\bm{c}^k)+{\bm{v}_i(\bm{c}^k)}^T\bm{v}_i(\bm{c}^k))}{2}},\; 1 \leq i \leq n.$$
	Note that, the formulations of $\boldsymbol{f}(\bm{c}^k)$ and $\boldsymbol{f}'(\bm{c}^k)$ involve computing the singular vectors $\{\bm{u}_i(\bm{c}^k)\}_{i=1}^n$ and $\{\bm{v}_i(\bm{c}^k)\}_{i=1}^n$ of $A(\bm{c}^k)$  incurs significant computational costs, particularly when $A(\bm{c}^k)$ is large-scale.
	A natural idea is to replace the singular vectors with approximate vectors  $\{\bm{u}_i^k\}_{i=1}^n$ (or $\{\overline{\bm{u}}_i^k\}_{i=1}^n$ ) and $\{\bm{v}_i^k\}_{i=1}^n$ (or $\{\overline{\bm{u}}_i^k\}_{i=1}^n$).	
	Therefore, in view of \eqref{ulmchebyshev}--\eqref{feq}, one may compute $\bm{c}^{k+1}$
	via the following two-step iterative procedure:
	\begin{equation}
		\left\{
		\begin{aligned}
			&	\overline{\bm{c}}^k = \bm{c}^k - B_k (J_k \bm{c}^k + \bm{b}^k), \nonumber\\
			&	\bm{c}^{k+1} =\overline{\bm{c}}^k- B_k\boldsymbol{\rho}^k, \label{itpro}\\
			&	B_{k+1} = B_k + B_k \left(2I - J_{k+1} B_k\right) \left(I - J_{k+1} B_k\right), \nonumber
		\end{aligned}
		\right.
	\end{equation}
	where $J_k,\; \bm{b}^k$ and $\boldsymbol{\rho}^k$ are respectively defined by
	\begin{align}
		[J_k]_{ij} &:= ({\bm{u}_i^k})^T A_j \bm{v}_i^k, \quad 1 \leq i,\;j \leq n,\label{defjk}\\
		[\bm{b}^k]_i &:= ({\bm{u}_i^k})^T A_0 \bm{v}_i^k-{\frac{\sigma_i^*(({\bm{u}_i^k})^T\bm{u}_i^k+({\bm{v}_i^k})^T\bm{v}_i^k)}{2}}, \quad 1 \leq i\leq n, \label{defbk}\\
		[\boldsymbol{\rho}^k]_i&:=({\overline{\bm{u}}_i^k})^T A(\overline{\bm{c}}_k) \overline{\bm{v}}_i^k-{\frac{\sigma_i^*(({\overline{\bm{u}}_i^k})^T\overline{\bm{u}}_i^k+({\overline{\bm{v}}_i^k})^T\overline{\bm{v}}_i^k)}{2}}, \quad 1 \leq i\leq n.\label{rho}
	\end{align}
	
	Overall, we propose the following Cayley-free two-step algorithm for solving the ISVP \eqref{iep}.
	\begin{breakablealgorithm}\label{improved algorithm Ulm}
		\caption{The Cayley-free two-step algorithm}
		\begin{algorithmic}[1]				
			\STATE Given $\bm{c}^0\in \mathbb{R}^{n}$ and $ B_0 \in \mathbb{R}^{n \times n} $, compute the singular values \(\{\sigma_i(\bm{c}^0)\}_{i=1}^n\), the normalized left singular vectors $\{\bm{u}_i(\bm{c}^0)\}_{i=1}^m$ and the normalized  right singular vectors \(\{\bm{v}_i(\bm{c}^0)\}_{i=1}^n\) of \(A(\bm{c}^0)\). Let	
			
			\begin{equation}
				U_0 := [\bm{u}_1^0, \bm{u}_2^0, \ldots, \bm{u}_m^0] = [\bm{u}_1(\bm{c}^0), \bm{u}_2(\bm{c}^0), \ldots, \bm{u}_m(\bm{c}^0)],\nonumber
			\end{equation}	
			\begin{equation}
				V_0 := [\bm{v}_1^0, \bm{v}_2^0, \ldots, \bm{v}_n^0] = [\bm{v}_1(\bm{c}^0), \bm{v}_2(\bm{c}^0), \ldots, \bm{v}_n(\bm{c}^0)],\nonumber
			\end{equation}	
			\begin{equation}
				\boldsymbol{\sigma}_0 := (\sigma_1^0, \sigma_2^0, \ldots, \sigma_n^0)^T = (\sigma_1(\bm{c}^0), \sigma_2(\bm{c}^0), \ldots, \sigma_n(\bm{c}^0))^T.\nonumber
			\end{equation}	
			Form the Jacobian matrix \( J_0:=f'(\bm{c}^0) \) and the vector \( \bm{b}^0 \) by
			\begin{equation}
				[J_0]_{ij} := ({\bm{u}_i^0})^T A_j \bm{v}_i^0,\quad
				1 \leq i,\;j \leq n, \nonumber
			\end{equation}
			\begin{equation}
				[\bm{b}^0]_i := ({\bm{u}_i^0})^T A_0 \bm{v}_i^0-{\frac{\sigma_i^*(({\bm{u}_i^0})^T\bm{u}_i^0+({\bm{v}_i^0})^T\bm{v}_i^0)}{2}},\quad
				1 \leq i \leq n.\nonumber
			\end{equation}
			
			\STATE For $k=0,1,\cdots $, until convergence, do.
			
			(1) Compute $\overline{\bm{c}}^k$ by
			\begin{equation}
				\overline{\bm{c}}^k = \bm{c}^k - B_k (J_k \bm{c}^k + \bm{b}^k).\nonumber
			\end{equation}
			
			(2) Compute $\overline{X}_k$ and  $\overline{Y}_k$ by \eqref{diag ox}, \eqref{diag oy} and \eqref{upox}--\eqref{oxfree}.

			(3) Compute the matrix $\overline{U}_k:=[\overline{\bm{u}}^k_{1},\;\overline{\bm{u}}^k_{2},\;\ldots,\;\overline{\bm{u}}^k_m]$ and $\overline{V}_k:=[\overline{\bm{v}}^k_{1},\;\overline{\bm{v}}^k_{2},\;\ldots,\;\overline{\bm{v}}^k_n]$ by \eqref{defouv}.
			
			(4) Compute  $\boldsymbol{\rho}^k$ by \eqref{rho}.
			
			(5)  Compute \( \bm{c}^{k+1} \) by
			\[
			\bm{c}^{k+1} = \overline{\bm{c}}^k - B_k\boldsymbol{\rho}^k.
			\]
			(6) Compute $\overline{E}_k$ and  $\overline{F}_k$ by \eqref{oe}--\eqref{upof}.
			
			(7) Compute  $U_{k+1}:=[\bm{u}_1^{k+1},\; \bm{u}_2^{k+1}, \;\ldots,\; \bm{u}_m^{k+1}]$ and $V_{k+1}:=[\bm{v}_1^{k+1}, \;\bm{v}_2^{k+1},\; \ldots, \;\bm{v}_n^{k+1}]$  by \eqref{uvk+1}.
			
			(8) Compute the approximate Jacobian matrix \( J_{k+1} \) and the vector \( \bm{b}^{k+1} \) by
			\[
			[J_{k+1}]_{ij} = ({\bm{u}_i^{k+1}})^T A_j \bm{v}_i^{k+1}, \quad 1 \leq i, \, j \leq n,
			\]
			\[
			[\bm{b}^{k+1}]_i = ({\bm{u}_i^{k+1}})^TA_0\bm{v}_i^{k+1}-\frac{\sigma_i^*(({\bm{u}_i^{k+1}})^T\bm{u}_i^{k+1}+({\bm{v}_i^{k+1}})^T\bm{v}_i^{k+1})}{2}, \quad
			1 \leq i \leq n.
			\]
			(9) Compute the Chebyshev matrix \( B_{k+1} \) by
			\[
			B_{k+1} = B_k + B_k \left(2I - J_{k+1} B_k\right) \left(I - J_{k+1} B_k\right).
			\]
		\end{algorithmic}
	\end{breakablealgorithm}
	
	\section{Convergence Analysis}\label{section3}
	\noindent In this section, we will carry out the convergence analysis for \autoref{improved algorithm Ulm}. For this purpose, we recall the notion of root-convergence rate which is known in \cite[Chapter 9]{iterative solution}.
	\begin{definition}\label{defroot}
		Let $\{\bm{x}^k\}$ be a sequence with the limit $\bm{x}^*$. Then, the numbers
		\begin{equation}\nonumber
			R_p\{\bm{x}^k\}=\left\{
			\begin{aligned}
				&\limsup\limits_{k\to\infty}\Vert\bm{x}^k-\bm{x}^*\Vert^{\frac{1}{k}}, \ if \ p=1,\\
				&\limsup\limits_{k\to\infty}\Vert\bm{x}^k-\bm{x}^*\Vert^{\frac{1}{p^k}}, \ if \ p>1.
			\end{aligned}
			\right.
		\end{equation}
		are the root-convergence factors of $\{\bm{x}^k\}$. The quantity
		\begin{equation}\nonumber
			O_R\{\bm{x}^*\}=\left\{
			\begin{aligned}
				&\infty         \ \ &&if \ R_p\{\bm{x}^k\}=0,\ \ \forall p\in [1,\infty), \\
				&\inf\{p\in [1,\infty)|\;R_p\{\bm{x}^k\}=1\}, \ \ &&otherwise.
			\end{aligned}
			\right.
		\end{equation}
		is called the root-convergence rate of $\{\bm{x}^k\}$.
	\end{definition}
	The following  lemma  presents a perturbation bound for the inverse; see for example \cite[pp. 74-75]{matrix compution}.
	\begin{lemma}\label{lemma AB-1}
		Let $ A,\ B\in\mathbb{R}^{n\times n} $. Suppose that $ B $ is nonsingular and $ \Vert B^{-1}\Vert\cdot\Vert A-B\Vert<1 $. Then, $ A $ is nonsingular and moreover
		\begin{equation*}
			\Vert A^{-1}\Vert \leq \frac{\Vert B^{-1}\Vert}{1-\Vert B^{-1}\Vert\cdot \Vert A-B\Vert}.
		\end{equation*}
	\end{lemma}
	
	The  following result can be found in  \cite[Fact 7]{error bound}.
	
	\begin{lemma}\label{Q-Q<A-A}
		Let $\mathcal{O}(n)$ denote the set of $n \times n$ orthogonal matrices, and let $\mathcal{D}$ denote the set of $m \times n$ real diagonal matrices with decreasing diagonal entries. Let  $A \in \mathbb{R}^{m \times n}$.
		There exist positive numbers $r_0$ and $\eta_0$ such that
		
		\begin{equation*}
			\min_{U\in \mathcal{O}(m), V \in \mathcal{O}(n), U^TAV \in \mathcal{D}
			}	\|U-\overline{U}\|_{F}+\|V-\overline{V}\|_{F}  \leq \eta_0\|A-\overline{A}\|_{F},
		\end{equation*}
		whenever $\overline{U}\in \mathcal{O}(m), \overline{V} \in \mathcal{O}(n), \overline{U}^T\overline{A}\; \overline{V} \in \mathcal{D}, \|A-\overline{A}\|_{F} \leq r_0. $
	\end{lemma}

	Recall that $ \bm{c}^* \in \mathbb{R}^n$ is a solution of ISVP \eqref{iep} with the singular values $ \left\{ \sigma_i^*\right\}_{i=1}^n $ satisfying \eqref{distinct}. Recall the definition of $\boldsymbol{f}'(\bm{c}^*) $, by \cite[Lemma 4]{new dist con} and the continuity of the matrix and its inverse, we have the  following lemma.
	\begin{lemma}\label{lemma j<k}
		Suppose that $ \boldsymbol{f}'(\bm{c}^*) $ is nonsingular. Then there exist  $r_1>0$ and $ \eta_1 \geq 1 $
		such that for each $ \bm{c}\in B(\bm{c}^*,r_1)$, $ {\boldsymbol{f}'(\bm{c})} $ is also nonsingular and
		$$ \Vert                 {\boldsymbol{f}'(\bm{c})}^{-1}       \Vert \leq \eta_1.$$
	\end{lemma}
	
	In the remainder, we assume that $\{\bm{c}^k\}$, $\{\overline{\bm{c}}^k\}$,  $\{\overline{E}_k\}$,  $\{\overline{F}_k\}$,  $\{\overline{X}_k\}$, $\{\overline{Y}_k\}$, $\{U_k\}$, $\{V_k\}$, $\{\overline{U}_k\}$ and $\{\overline{V}_k\}$ are generated by \autoref{improved algorithm Ulm}.  Recall that, for each $k\in\mathbb{N}$, the matrices $U_k^*$, $\overline{U}_{k}^*$, $X_k$,  $Y_k$, $E_k$ and $F_k$ are defined by \eqref{minu^*}, \eqref{defuv}, \eqref{defzn} and \eqref{minou^*}.
	Define $\{G_{k}\}$, $\{H_k\}\subset\mathbb{R}^{m\times m} $  by
	\begin{equation}\label{HHdef}
		\overline{U}_k=U_{k}^*(I+G_k)\quad {\rm and}\quad U_{k+1}=\overline{U}_{k}^*(I+H_{k})\quad\text{for each} \;k=0,\;1,\;\ldots.
	\end{equation}
	Then we have the following three lemmas. In particular,  \autoref{lemmasym} below is a direct consequence of \cite[Lemma 1]{ashima iep mu} as, by letting $U(\bm{c}^*)\in \mathcal{U} $,
	$$A(\bm{c}^*)A(\bm{c}^*)^T=U(\bm{c}^*)\operatorname{diag}((\sigma_{1}^{*})^2, \ldots, (\sigma_{n}^{*})^2,0, \ldots,0) U(\bm{c}^*)^T.$$
	While the proof of  \autoref{3times}  is similar to that of \cite[Lemma 2]{ashima iep mu}. For the sake of completeness, we still provide its proof here.
	
	\begin{lemma}\label{lemmasym}Let $k\in\mathbb{N}$.
		For each $(i, j)\in\mathcal{I}_4$, it holds that $[X_k]_{ij}=[X_k]_{ji}$.
	\end{lemma}

	\begin{lemma}\label{3times} Let $k\in\mathbb{N}$. Then,  the following inequalities hold:
		\begin{equation}\label{xk+1}
			\|E_k\| \leq 3\|G_{k}\|		\quad{\rm and}\quad \|X_{k+1}\| \leq 3\|H_{k}\|.
		\end{equation}
	\end{lemma}
	\begin{proof}
		We only need to prove the first inequality in \eqref{xk+1} as the proof of the second one is similar. For this end, we recall that $\overline{U}_k=[\overline{\bm{u}}_1^{k}, \overline{\bm{u}}_2^{k}, \ldots, \overline{\bm{u}}_m^{k}]$ and $U_k^*=[\bm{u}_1(\bm{c}^*),\ldots,\;\bm{u}_n(\bm{c}^*),\bm{u}_{n+1}^k(\bm{c}^*),\ldots,\bm{u}_{m}^k(\bm{c}^*)]$ is defined by \eqref{minu^*}. Write
		\begin{equation}\label{um-n}
			\overline{U}^k_{m-n}:=[\overline{\bm{u}}^k_{n+1},\;\overline{\bm{u}}^k_{n+2},\;\ldots,\;\overline{\bm{u}}^k_m]
		\end{equation}	and
		\begin{equation}\label{um-n^*}U_{m-n}^k(\bm{c}^*):=[\bm{u}_{n+1}^k(\bm{c}^*),\;\bm{u}_{n+2}^k(\bm{c}^*),\;\ldots,\;\bm{u}_m^k(\bm{c}^*)].	\end{equation}
		Then, the problem \eqref{minou^*} is reduced to the following optimization problem:
		\begin{equation}\label{optimm-n}
			\min_{Q^TQ=I}\|U_{m-n}^k(\bm{c}^*)Q-\overline{U}_{m-n}^k\|_F.
		\end{equation}
		%
		%
		%
		We can obtain the optimal $\overline{Q}_k$ of \eqref{optimm-n} by the polar decomposition
		\begin{equation}\label{polarsub}
			(U_{m-n}^k(\bm{c}^*))^{T}\overline{U}^k_{m-n}=\overline{Q}_k\overline{P}_k,
		\end{equation}
		where $\overline{Q}_k$ is an orthogonal matrix, and $\overline{P}_k$ is a symmetric and positive semi-deﬁnite matrix (see \cite[\S 6.4.1]{matrix compution}).
		Thus, by the definitions of $\overline{U}_k^*$ and $U_{k}^*$, one has
		\begin{equation}\label{U^*=}
			\overline{U}_k^*=[\bm{u}_{1}(\bm{c}^*),\;\bm{u}_{1}(\bm{c}^*),\;\ldots,\;\bm{u}_{n}(\bm{c}^*),\;U_{m-n}^k(\bm{c}^*)\overline{Q}_k]
			=U_{k}^*
			\begin{pmatrix}
				I & 0 \\
				0 & \overline{Q}_k
			\end{pmatrix}.						
		\end{equation}
		To proceed, we define $$Q_k:=\begin{pmatrix}
			I & 0 \\
			0 & \overline{Q}_k
		\end{pmatrix},\quad P_k:=\begin{pmatrix}
			I & 0 \\
			0 & \overline{P}_k
		\end{pmatrix},$$ and a block diagonal matrix $G^k_{diag}$ by					
		\begin{equation*}
			[G^k_{diag}]_{ij}:=\left\{\begin{array}{ll}
				[G_k]_{ij}, & \text{$n+1 \leq i,j \leq m$} \\
				0, & \text{otherwise}
			\end{array}.\right .
		\end{equation*}					
		Then, thanks to \eqref{U^*=}, one has
		\begin{equation}\label{over}
			\overline{U}^*_k=U^*_kQ_k.
		\end{equation}
		In addition, by	\eqref{HHdef}, \eqref{um-n}, \eqref{um-n^*}, \eqref{polarsub} and the definition of $G^k_{diag}$, one checks 	
		\begin{equation}\label{polarofi+x}
			I+G^k_{diag}=Q_kP_k.
		\end{equation}
		Note by \eqref{defzn}, \eqref{HHdef} and \eqref{over} that
		$$E_k= {\overline{U}^*_k}^T\overline{U}_k-I \nonumber
		= Q_k^T{U_k^*}^T\overline{U}_k-I \nonumber
		= Q_k^T(I+G_k)-I,$$
		which together with \eqref{polarofi+x} gives
		\begin{align}\label{X=}
			E_k= Q_k^T(G^k-G^k_{diag}+Q_kP_k)-I= Q_k^T(G^k-G^k_{diag})+P_k-I.
		\end{align}
		Since $\|P_k\|=\|I+G^k_{diag}\|$ (by \eqref{polarofi+x}),
		all eigenvalues of $P_k$ range over the interval \([1-\|G^k_{diag}\|,1+\|G^k_{diag}\|]\).
		Thus, we get
		\begin{equation}\label{P-I}
			\|P_k-I\|\leq\|G^k_{diag}\|.
		\end{equation}
		Moreover, by the definition of $G^k_{diag}$, one sees $\|G^k_{diag}\|\leq\|G^k\|.$		
		Combining this with  \eqref{X=} and \eqref{P-I}, we deduce
		\begin{equation}
			\|E_k\|\leq\|Q_k^T(G^k-G^k_{diag})\|+\|P_k-I\|\leq 3\|G^k\|. \nonumber
		\end{equation}
		The proof is complete.
	\end{proof}

	\begin{lemma}\label{lemma impront}
		There exists  a constant $\eta_2 >1$ such that for each \( k\in\mathbb{N} \), the following assertions hold:
		\begin{itemize}
			\item [$\rm(i)$]If $ \max\left\{\Vert X_k \Vert,\Vert Y_k \Vert\right\} \leq 1 $, then
			\begin{align}
				\max\left\{ \|X_k - \overline{X}_k\| ,\|Y_k - \overline{Y}_k\| \right\}
				\leq
				\eta_2({\Vert X_k \Vert}^2+{\Vert Y_k \Vert}^2 + \|\overline{\bm{c}}^k - \bm{c}^*\|),\label{xk-xkover}\\
				\max\left\{ \|E_k\|,\|F_k\| \right\}
				\leq
				3\eta_2({\Vert X_k \Vert}^2+{\Vert Y_k \Vert}^2 + \|\overline{\bm{c}}^k - \bm{c}^*\|);\label{ekfk}
			\end{align}
			\item [$\rm(ii)$]If $\max\left\{\Vert E_k \Vert,\Vert F_k \Vert\right\} \leq 1,$ then
			\begin{align}
				\max\left\{\|E_k - \overline{E}_k\| ,\|F_k - \overline{F}_k\| \right\}
				\leq
				\eta_2 (
				{\Vert E_k \Vert}^2+{\Vert F_k \Vert}^2+\|\bm{c}^{k+1} - \bm{c}^*\|
				), \label{ek-overek}\\
				\max\left\{ \|X_{k+1}\|,\|Y_{k+1}\| \right\}
				\leq
				3\eta_2 (
				{\Vert E_k \Vert}^2+{\Vert F_k \Vert}^2+\|\bm{c}^{k+1} - \bm{c}^*\|
				).\label{xk+1yk+1}
			\end{align}	
		\end{itemize}
	\end{lemma}
	\begin{proof}
		For simplicity, we define
		\begin{equation*}
			r_{\min}^*:=\min\limits_{1\leq i \leq n}\left\{\sigma_i^*-\sigma_{i+1}^*\right\}\quad{\rm and}\quad	\tilde{\eta} := \frac{4m}{r_{\min}^*} \max \left\{ \|{\Sigma}^*\|, \, \max_j \|A_j\| \right\},
		\end{equation*}
		where   we adopt the convention $\sigma_{n+1}^*:=0$.
		Set $\eta_2:=1+2\tilde{\eta}.$ Below we shall show that $\eta_2$ is as desired. For this end, let $k \in \mathbb{N}$.
		
		(i) Suppose that $ \max\left\{\Vert X_k \Vert,\Vert Y_k \Vert\right\} \leq 1 $. We first give  the estimations  for $\|X_k-\overline{X}_k\|$ and $\|Y_k-\overline{Y}_k\|$. In fact, by \eqref{u1}, \eqref{v1}, \eqref{ox1} and  \eqref{oy1}, one has
		\begin{equation}\label{x-ox}	
			(X_k-\overline{X}_k)^T+(X_k-\overline{X}_k)+X_k^T X_k=0	
		\end{equation}
		and
		\begin{equation}\label{y-oy}
			(Y_k-\overline{Y}_k)^T+(Y_k-\overline{Y}_k)+Y_k^T Y_k=0,
		\end{equation}
		which gives
		\begin{equation}\label{diagx-ox}
			[X_k-\overline{X}_k]_{ii}=-\frac{[X_k^T X_k]_{ii}}{2},\quad 1 \leq i \leq m
		\end{equation}
		and
		\begin{equation}\label{diagy-oy}
			[Y_k-\overline{Y}_k]_{ii}=-\frac{[Y_k^T Y_k]_{ii}}{2}, \quad 1 \leq i \leq n.
		\end{equation}
		Note by \autoref{lemmasym} and \eqref{oxfree} that $[X_k-\overline{X}_k]_{ij}=[X_k-\overline{X}_k]_{ji} $ for each $ (i,j) \in \mathcal{I}_4$. Then, it follows  from \eqref{x-ox} that
		\begin{equation}\label{x-oxi4}
			[X_k-\overline{X}_k]_{ij} = -\frac{[X_k^TX_k]_{ij}}{2},\quad (i,j) \in \mathcal{I}_4.
		\end{equation}
		To proceed, we write for simplicity $T_k:=U_{k}^{T}(A(\bm{c}^{*}) - A(\overline{\bm{c}}^{k}))V_{k}-X_{k}^T \Sigma^{*} Y_{k}$.
		Then, thanks to \eqref{c^*1} and \eqref{oc1}, we have
		\begin{equation*}
			[T_k]_{ij} = [(X_{k}-\overline{X}_{k})^{T}\Sigma^{*}+\Sigma^{*}(Y_{k} - \overline{Y}_{k})]_{ij},\quad(i,j)\in \mathcal{I}_1\cup\mathcal{I}_2.
		\end{equation*}	
		This together with \eqref{x-ox} and  \eqref{y-oy} respectively gives
		\begin{equation}\label{cx-ox}
			[T_{k}+X_k^T X_k\Sigma^*]_{ij}
			= [\Sigma^{*}(Y_{k} - \overline{Y}_{k})-(X_{k}-\overline{X}_{k})\Sigma^{*} ]_{ij},\quad(i,j)\in \mathcal{I}_1\cup\mathcal{I}_2
		\end{equation}
		and
		\begin{equation}\label{cy-oy}
			[T_{k}+\Sigma^*Y_k^T Y_k]_{ij}= [(X_k-\overline{X}_k)^T\Sigma^*- \Sigma^*(Y_k-\overline{Y}_k)^T]_{ij},\quad(i,j)\in \mathcal{I}_1\cup\mathcal{I}_2.
		\end{equation}
		Thus, by solving the equations in \eqref{cx-ox} and \eqref{cy-oy} with $(i,j)\in \mathcal{I}_1$,  we  get	
		\begin{align}\label{x-oxi1}
			[X_k - \overline{X}_k]_{ij}
			=\frac{\sigma_j^*[T_{k}+X_k^T X_k\Sigma^{*}]_{ij}+\sigma_i^*[T_{k}+\Sigma^*Y_k^T Y_k]_{ji}}{(\sigma_i^*)^2-(\sigma_j^*)^2},\quad(i,j)\in \mathcal{I}_1,
		\end{align}
		\begin{align}\label{y-oyi1}
			[Y_k - \overline{Y}_k]_{ij}
			=\frac{ \sigma_i^*[T_{k}+X_k^T X_k\Sigma^{*}]_{ij}+\sigma_j^*[T_{k}+\Sigma^*Y_k^T Y_k]_{ji}}{(\sigma_i^*)^2-(\sigma_j^*)^2},\quad(i,j)\in \mathcal{I}_1.
		\end{align}	
		Since $[\Sigma^{*}(Y_{k} - \overline{Y}_{k})]_{ij}=0$ for each $(i,j) \in \mathcal{I}_2$,
		we have by  \eqref{cx-ox} that
		\begin{align}\label{x-oxi2}
			[X_k - \overline{X}_k]_{ij} =\frac{[T_k+X_k^T X_k\Sigma^*]_{ij}}{-\sigma_j^*},	\quad(i,j) \in \mathcal{I}_2.
		\end{align}
		Recall from \eqref{x-ox}  that
		\begin{equation}\label{ld x-ox}
			-[X_k^TX_k]_{ij}=[X_k-\overline{X}_{k}]_{ij}+[X_k-\overline{X}_{k}]_{ji},\quad (i,j) \in \mathcal{I}_2.
		\end{equation}
		Then, substituting  \eqref{x-oxi2} into \eqref{ld x-ox}, we have
		\begin{align}\label{x-oxi3}
			[X_k - \overline{X}_k]_{ij} =\frac{[T_k]_{ji}}{\sigma_i^*},	\quad (i,j) \in \mathcal{I}_3.
		\end{align}
		By \eqref{defuv}, \eqref{A} and the definition of $T_k$,  one checks
		\begin{equation*}
			\|T_k\|\le4\max\limits_j \|A_j\| \cdot \|\overline{\bm{c}}^k - \bm{c}^*\| +\|\Sigma^*\|({\|X_k\|}^2+{\|Y_k\|}^2)
		\end{equation*}
		(recalling  that $ \max\left\{\Vert X_k \Vert,\Vert Y_k \Vert\right\} \leq 1 $ and that $U_k^*$ and $V(\bm{c}^*)$ are orthogonal). Thus, noting
		the definition of $r_{\min}^*$ and using \eqref{diagx-ox}, \eqref{x-oxi4}, \eqref{x-oxi1}, \eqref{x-oxi2} and \eqref{x-oxi3}, we have by simple calculations that
		\begin{align*}
			|[X_k - \overline{X}_k]_{ii}|
			\leq \frac{{\|X_k\|}^2}{2}
			\leq \frac{\|\Sigma^*\|\cdot{\|X_k\|}^2}{r_{\min}^*},\quad 1 \leq i \leq m,
		\end{align*}		
		\begin{equation*}
			|[X_k-\overline{X}_k]_{ij}| \leq \frac{{\|X_k\|}^2}{2}
			\leq \frac{\|\Sigma^*\|\cdot{\|X_k\|}^2}{r_{\min}^*},\quad (i,j) \in \mathcal{I}_4,
		\end{equation*}.
		\begin{align}
			|[X_k - \overline{X}_k]_{ij}|
			&\leq \frac{\|T_k\|+\|\Sigma^*\|({\|X_k\|}^2+{\|Y_k\|}^2)}{|\sigma_i^*-\sigma_j^*|}\nonumber\\
			&\leq \frac{4\max\limits_j \|A_j\| \cdot \|\overline{\bm{c}}^k - \bm{c}^*\| +2\|\Sigma^*\|({\|X_k\|}^2+{\|Y_k\|}^2)}{r_{\min}^*},\quad (i,j) \in \mathcal{I}_1,\nonumber
		\end{align}
		
		\begin{align*}
			|[X_k - \overline{X}_k]_{ij}|
			&\leq \frac{\|T_k\|+\|\Sigma^*\|\cdot{\|X_k\|}^2}{r_{\min}^*}\nonumber\\
			&\leq \frac{4\max\limits_j \|A_j\| \cdot \|\overline{\bm{c}}^k - \bm{c}^*\| +\|\Sigma^*\|(2{\|X_k\|}^2+{\|Y_k\|}^2)}{r_{\min}^*},\quad (i,j) \in \mathcal{I}_2,
		\end{align*}
		\begin{align*}
			|[X_k - \overline{X}_k]_{ij}|
			&\leq \frac{4\max\limits_j \|A_j\| \cdot \|\overline{\bm{c}}^k - \bm{c}^*\| +\|\Sigma^*\|({\|X_k\|}^2+{\|Y_k\|}^2)}{r_{\min}^*},\quad (i,j) \in \mathcal{I}_3.
		\end{align*}
		Hence, we conclude
		\begin{align}\label{allx-ox}
			\lVert X_k - \overline{X}_k \rVert &\leq m \max_{1 \leq i,\;j \leq m} \left| [X_k - \overline{X}_k]_{ij} \right|\nonumber\\
			&\leq m\frac{4\max\limits_j \|A_j\| \cdot \|\overline{\bm{c}}^k - \bm{c}^*\| +2\|\Sigma^*\|({\|X_k\|}^2+{\|Y_k\|}^2)}{r_{\min}^*}\nonumber\\
			&\leq \tilde{\eta} ({\|X_k\|}^2+{\|Y_k\|}^2 + \|\overline{\bm{c}}^k - \bm{c}^*\|),
		\end{align}
		where the last inequality holds because of the definition of 	$\tilde{\eta} $. On the other hand, using a similar arguments for \eqref{allx-ox}, we have by \eqref{diagy-oy} and \eqref{y-oyi1} that
	\begin{align}
		\lVert Y_k - \overline{Y}_k \rVert 	\leq \tilde{\eta} ({\|X_k\|}^2+{\|Y_k\|}^2 + \|\overline{\bm{c}}^k - \bm{c}^*\|).\nonumber
	\end{align}
	Therefore, \eqref{xk-xkover} is seen to hold as $\tilde{\eta}\le \eta_2 $	by definition. Now we prove \eqref{ekfk}.
	By the first inequalities in \eqref{defuv} and \eqref{defouv}, we get
	\begin{equation*}
		\overline{U}_{k} = U_{k}^{*} (I + X_{k}) (I - \overline{X}_{k}).
	\end{equation*}
	Then, recalling that $\overline{U}_k=U_{k}^*(I+G_k)$, we have
	\begin{equation*}
		G_k = X_k - \overline{X}_k - X_k \overline{X}_k = (I+X_k)(X_k - \overline{X}_k) - X_k^2.
	\end{equation*}
	Thus, noting that $\max\left\{\Vert X_k \Vert,\Vert Y_k \Vert\right\} \leq 1$, one has by \eqref{allx-ox}  and the definition of $\eta_2$ that
	\begin{align}\label{esalle}
		\|	G_k\| &\leq 2 \|X_k - \overline{X}_k\| + \|X_k\|^2 \nonumber
		\leq \eta_2 ({\|X_k\|}^2+{\|Y_k\|}^2+ \|\overline{\bm{c}}^k - \bm{c}^*\|).
	\end{align}
	Similarly, using \eqref{defuv},  \eqref{defouv} and the second equality in \eqref{defzn}, one can easily prove
	\begin{equation*}
		\|F_k\| \leq \eta_2 ({\|X_k\|}^2+{\|Y_k\|}^2+ \|\overline{\bm{c}}^k - \bm{c}^*\|).
	\end{equation*}
	Therefore, applying \autoref{3times}, one sees that assertion (i) holds.	
	
	(ii)	We write for simplicity $\overline{T}_k:=\overline{U}_k^{T}(A(\bm{c}^{*}) - A(\bm{c}^{k+1}))\overline{V}_k-E_{k}^T \Sigma^{*} F_{k}$. Then, by \eqref{defzn} and \eqref{A},  one checks
	\begin{equation}\label{oTkest}
		\|\overline{T}_k\|\le4\max\limits_j \|A_j\| \cdot \|\overline{\bm{c}}^k - \bm{c}^*\| +\|\Sigma^*\|({\|E_k\|}^2+{\|F_k\|}^2)
	\end{equation}
	(recalling  that $ \max\left\{\Vert E_k \Vert,\Vert F_k \Vert\right\} \leq 1 $ and that $\overline{U}_k^*$ and $V(\bm{c}^*)$ are orthogonal).
	By \eqref{z1}--\eqref{ck+1}, we get
	\begin{equation}\label{z-oz}	
		(E_k-\overline{E}_k)^T+(E_k-\overline{E}_k)+E_k^T E_k=0,
	\end{equation}
	\begin{equation}\label{n-on}
		(F_k-\overline{F}_k)^T+(F_k-\overline{F}_k)+F_k^T F_k=0
	\end{equation}
	and
	\begin{equation}\label{ck+1-c^*}
		[\overline{T}_k]_{ij} = [(E_{k}-\overline{E}_{k})^{T}\Sigma^{*}+\Sigma^{*}(F_{k} - \overline{F}_{k})]_{ij},\quad (i,j) \in \mathcal{I}_1\cup\mathcal{I}_2.
	\end{equation}
	Similar to the arguments for \eqref{diagx-ox}--\eqref{x-oxi4}, \eqref{x-oxi1}--\eqref{x-oxi2} and \eqref{x-oxi3}, we get from \eqref{z-oz}--\eqref{ck+1-c^*} that

	\begin{equation*}
		[E_k - \overline{E}_k]_{ii}=-\frac{[E_k^T E_k]_{ii}}{2},\quad 1 \leq i \leq m,
	\end{equation*}
	
	\begin{equation*}
		[E_k - \overline{E}_k]_{ij}=\frac{\sigma_i^*[\overline{T}_k+\Sigma^*F_k^T F_k]_{ji}+\sigma_j^*[\overline{T}_k+E_k^T E_k\Sigma^*]_{ij}}{(\sigma_i^*)^2-(\sigma_j^*)^2},\quad (i,j) \in \mathcal{I}_1	,
	\end{equation*}
	
	\begin{equation*}
		[E_k - \overline{E}_k]_{ij}=\frac{[\overline{T}_k+E_k^TE_k\Sigma^*]_{ij}}{-\sigma_j^*},  \quad(i,j) \in \mathcal{I}_2,
	\end{equation*}
	\begin{equation*}
		[E_k - \overline{E}_k]_{ij}=	\frac{[\overline{T}_k]_{ji}}{\sigma_i^*},\quad (i,j) \in \mathcal{I}_3,
	\end{equation*}
	\begin{equation*}
		|	[E_k - \overline{E}_k]_{ij} | = -\frac{[E_k^TE_k]_{ij}}{2},\quad (i,j) \in \mathcal{I}_4,
	\end{equation*}
	\begin{equation*}
		[F_k - \overline{F}_k]_{ii}=	-\frac{[F_k^T F_k]_{ii}}{2},\quad 1 \leq i \leq n	,
	\end{equation*}
	
	\begin{equation*}
		[F_k - \overline{F}_k]_{ij}=\frac{\sigma_i^*[\overline{T}_k+E_k^T E_k\Sigma^*]_{ij}+\sigma_j^*[\overline{T}_k+\Sigma^*F_k^T F_k]_{ji}}{(\sigma_i^*)^2-(\sigma_j^*)^2}, \quad \quad (i,j) \in \mathcal{I}_1.
	\end{equation*}
	Thus noting \eqref{oTkest} and using similar arguments for \eqref{allx-ox}, we can easily prove
	\begin{align}\label{e-overf}
		\lVert E_k - \overline{E}_k \rVert \leq \tilde{\eta} ({\|E_k\|}^2+{\|F_k\|}^2 + \|\bm{c}^{k+1} - \bm{c}^*\|)
	\end{align}	
	and
	\begin{align*}
		\lVert F_k - \overline{F}_k \rVert\leq \tilde{\eta}({\|E_k\|}^2+{\|F_k\|}^2 + \|\bm{c}^{k+1} - \bm{c}^*\|).
	\end{align*}
	Hence, \eqref{ek-overek} is seen to hold (noting $ \tilde{\eta}\le\eta_2$). It remains to prove  \eqref{xk+1yk+1}. For brevity,  we only give the estimation of $\|X_{k+1}\|$
	as the one for $\|Y_{k+1}\|$ is similar. Note  by the first equalities in \eqref{defzn} and \eqref{uvk+1} that
	\begin{equation*}
		{U}_{k+1} = \overline{U}_k^* (I + E_{k}) (I - \overline{E}_{k}).
	\end{equation*}
	Then, recalling that $U_{k+1}=\overline{U}_{k}^*(I+{H}_{k})$, we have
	\begin{equation*}
		H_k = E_k - \overline{E}_k - E_k \overline{E}_k = (I+E_k)(E_k - \overline{E}_k) - E_k^2.
	\end{equation*}
	Thus, noting that $\max\left\{\Vert E_k \Vert,\Vert F_k \Vert\right\} \leq 1$, one has by \eqref{e-overf}  that
	\begin{align*}
		\|{H}_{k}\| \leq \eta_2 ( {\|E_k\|}^2+{\|F_k\|}^2+ \|\bm{c}^{k+1} - \bm{c}^*\|).
	\end{align*}
	Hence, thanks to \autoref{3times}, we obtain
	\begin{equation*}
		\|X_{k+1}\| \leq 3\|{H}_{k}\| \leq 3\eta_2 ( {\|E_k\|}^2+{\|F_k\|}^2+ \|\bm{c}^{k+1} - \bm{c}^*\|).
	\end{equation*}	
	Therefore, assertion (ii) holds and the proof is complete.
\end{proof}

Recall that $\{\overline{U}_k=[\bm{u}_1^k,\;\ldots,\;\bm{u}_m^k]\}$ and $\{\overline{V}_k=[\bm{v}_1^k,\;\ldots,\;\bm{v}_n^k]\}$ are generated by \autoref{improved algorithm Ulm}.
For each $k \in \mathbb{N}$, we define $\overline{J}_k\in\mathbb{R}^{n\times n}$ and $\overline{\boldsymbol{b}}^k\in\mathbb{R}^{n}$ by
\begin{align*}
	&[\overline{{J}}_k]_{ij}:=({\overline{\bm{u}}_i^k})^TA_j\overline{\bm{v}}_i^k,\quad1\leq i,\;j \leq n,\\
	&[\overline{\bm{b}}^k]_i:=({\overline{\bm{u}}_i^k})^T A_0 \overline{\bm{v}}_i^k-{\frac{\sigma_i^*(({\overline{\bm{u}}_i^k})^T\overline{\bm{u}}_i^k+({\overline{\bm{v}}_i^k})^T\overline{\bm{v}}_i^k)}{2}},\quad1\leq i \leq n .
\end{align*}
Recall also that $\{J_k\}$ and $\{\bm{b}^k\}$ are defined by \eqref{defjk} and \eqref{defbk}, respectively.
\begin{lemma}\label{esalljk+bk}
	For each $k \in \mathbb{N}$, we have
	\begin{equation}\label{esjk+bk}
		\Vert J_k\bm{c}^*+\bm{b}^k\Vert \leq 2\sqrt{n}\Vert \Sigma^* \Vert \left\{{\|X_k\|}^2,{\|Y_k\|}^2 \right\}
	\end{equation}
	and
	\begin{equation}\label{esojk+obk}
		\Vert \overline{J}_k\bm{c}^*+\overline{\bm{b}}^k\Vert
		\leq 2\sqrt{n}\Vert \Sigma^* \Vert\left\{{\|E_k\|}^2,{\|F_k\|}^2 \right\}.
	\end{equation}
	\begin{proof}
		We only give the proof of \eqref{esjk+bk} as the one for \eqref{esojk+obk} is similar.
		Note  by \eqref{u1}--\eqref{c^*1} that
		\begin{equation}\label{ac*-uksigma}
			U_k^TA(\bm{c}^*)V_k-U_k^TU_k\Sigma^*=\Sigma^*Y_k-X_k\Sigma^*+X_k^T\Sigma^*Y_k-X_k^TX_k\Sigma^*
		\end{equation}
		and
		\begin{equation}\label{ac*-vksigma}
			U_k^TA(\bm{c}^*)V_k-\Sigma^*V_k^TV_k=X_k^T\Sigma^*-\Sigma^*Y_k^T+X_k^T\Sigma^*Y_k-\Sigma^*Y_k^TY_k.
		\end{equation}
		Adding the diagonal terms of \eqref{ac*-uksigma} and \eqref{ac*-vksigma}, we have, for each $1\le i\le n$,
		\begin{equation*} 
			({\bm{u}_i^k})^TA(\bm{c}^*)\bm{v}_i^k-\frac{\sigma_i^*}{2}(({\bm{u}_i^k})^T\bm{u}_i^k+({\bm{v}_i^k})^T\bm{v}_i^k)=[X_k^T\Sigma^*Y_k-\frac{X_k^TX_k\Sigma^*+\Sigma^*Y_k^TY_k}{2}]_{ii},
		\end{equation*}
		and so
		\begin{align*}
			\vert ({\bm{u}_i^k})^TA(\bm{c}^*)\bm{v}_i^k-\frac{\sigma_i^*}{2}(({\bm{u}_i^k})^T\bm{u}_i^k+({\bm{v}_i^k})^T\bm{v}_i^k)\vert
			&\leq 2\Vert \Sigma^*\Vert\max\left\{{\|X_k\|}^2,{\|Y_k\|}^2 \right\}.
		\end{align*}
		Thus, \eqref{esjk+bk} is seen to  hold as, by  the definitions of $J_k$ and $\bm{b}^k$,
		\begin{equation}\label{jk+bk}
			[J_k\bm{c}^*+\bm{b}^k]_i= ({\bm{u}_i^k})^TA(\bm{c}^*)\bm{v}_i^k-\frac{\sigma_i^*}{2}(({\bm{u}_i^k})^T\bm{u}_i^k+({\bm{v}_i^k})^T\bm{v}_i^k), \quad	1 \leq i \leq n.\nonumber
		\end{equation}
		The proof is complete.
	\end{proof}
\end{lemma}

\begin{lemma}\label{esj_x-j_y}
	Let $l$ be a positive integer.	 Assume that $ \max\left\{\Vert X_k \Vert,\Vert Y_k \Vert\right\} \leq 1 $ and $\max\left\{\Vert E_k \Vert,\Vert F_k \Vert\right\} \leq 1$ hold for each $ 0 \leq k \leq l.$
	Then,  it holds that
	\begin{equation}
		\|J_l-J_0\|
		\leq 8n \max\limits_j\Vert A_j \Vert\sum\limits_{k=0}^{l-1}(\max\left\{\|\overline{X}_k\|, \|\overline{Y}_k\| \right\} + \max\left\{\|\overline{Y}_k\|, \|\overline{F}_k\| \right\})\label{esj_l-j_0}
	\end{equation}	
	and
	\begin{equation}\label{esj_l-oj_l}
		\Vert J_l- \overline{J}_l \Vert\leq 8n\max\limits_j\Vert A_j \Vert\cdot\max\left\{ \Vert \overline{X}_l\Vert,\Vert \overline{Y}_l \Vert \right\}.
	\end{equation}	
	Additionally, if $ \max\left\{\Vert X_{l+1} \Vert,\Vert Y_{l+1} \Vert\right\} \leq 1,$ then
	\begin{equation}\label{esj_l+1-j_l}
		\Vert J_{l+1}-J_l \Vert \leq 8n \max\limits_j\Vert A_j \Vert(\max\left\{ \|\overline{X}_k\|,\|\overline{Y}_k\| \right\}+\max\left\{ \|\overline{E}_k\|,\|\overline{F}_k\| \right\}).
	\end{equation}
\end{lemma}
\begin{proof}	
	Recalling that $\{U_k^*\}$, $\{\overline{U}_k^*\}$ and $V(\bm{c}^*)$ are orthogonal, we have by \eqref{defuv} and \eqref{defzn} that
	\begin{equation}\label{ukvk}
		\|U_{k}\|	 \leq 1+\|X_k\|,\quad\|V_{k}\|	 \leq 1+\|Y_k\|, \quad k\in\mathbb{N}
	\end{equation}
	and
	\begin{equation*}
		\|\overline{U}_{k}\|	 \leq 1+\|E_k\|,\quad\|\overline{V}_{k}\|	 \leq 1+\|F_k\|, \quad k\in\mathbb{N}.
	\end{equation*}
	Then, thanks to the assumptions, one sees
	\begin{equation}\label{uvuv}
		\max\left\{ \|U_{k}\| ,\;\|\overline{U}_{k}\|,\; \|V_{k}\| ,\;\|\overline{V}_{k}\|  \right\} \leq 2,\quad 0\leq k\leq l.
	\end{equation}
	Moreover, by \eqref{defouv} and \eqref{uvk+1}, one has that for each $ k\in\mathbb{N}$,
	\begin{equation}\label{uk-ukover}
		\overline{U}_k-U_{k}=-U_{k}\overline{X}_{k},\quad\overline{V}_k-V_{k}=-V_{k}\overline{Y}_{k}
	\end{equation}
	and
	\begin{equation*}
		U_{k+1}-\overline{U}_k=-\overline{U}_k\overline{E}_k, \quad V_{k+1}-\overline{V}_k=-\overline{V}_k\overline{F}_k.
	\end{equation*}
	Thus, using  \eqref{uvuv}, we deduce
	\begin{equation}\label{esuk+k-uk}
		\Vert U_{k+1}-U_k\Vert\leq\Vert U_{k+1}-\overline{U}_k\Vert+\Vert \overline{U}_k-U_{k}\Vert\leq 2\left(\Vert \overline{X}_k\Vert+\Vert \overline{E}_{k}\Vert \right),\quad 0\leq k\leq l
	\end{equation}
	and
	\begin{equation}\label{esvk+k-vk}
		\Vert V_{k+1}-V_k\Vert\leq\Vert V_{k+1}-\overline{V}_k\Vert+\Vert \overline{V}_k-V_{k}\Vert\leq 2\left(\Vert \overline{Y}_k\Vert+\Vert \overline{F}_{k}\Vert \right),\quad 0\leq k\leq l.
	\end{equation}
	This gives
	\begin{equation}\label{Ul}
		\Vert U_l-U_0\Vert
		\leq \sum\limits_{k=0}^{l-1}\Vert U_{k+1}-U_k\Vert
		\leq  2\sum\limits_{k=0}^{l-1}(\|\overline{X}_k\|+\|\overline{E}_k\|)	
	\end{equation}	
	and
	\begin{equation}\label{Vl}
		\Vert V_l-V_0\Vert
		\leq \sum\limits_{k=0}^{l-1}\Vert V_{k+1}-V_k\Vert
		\leq  2\sum\limits_{k=0}^{l-1}(\|\overline{Y}_k\|+\|\overline{F}_k\|).		
	\end{equation}	
	By \eqref{defjk} and \eqref{uvuv}, we have
	\begin{equation*}
		\begin{split}
			\vert [J_l-J_0]_{ij}\vert &=\vert ({\bm{u}_i^l})^TA_j\bm{v}_i^l-({\bm{u}_i^0})^TA_j\bm{v}_i^0\vert\\
			&=\vert (\bm{u}_i^l-\bm{u}_i^0)^TA_j\bm{v}_i^l+({\bm{u}_i^0})^TA_j(\bm{v}_i^l-\bm{v}_i^0)\vert\\
			&\leq 2\Vert A_j\Vert(\Vert\bm{u}_i^l-\bm{u}_i^0\Vert+\Vert \bm{v}_i^l-\bm{v}_i^0\Vert),\quad 1\leq i,\ j\leq n.
		\end{split}
	\end{equation*}
	Combining this with \eqref{Ul} and \eqref{Vl}, one has
	\begin{align}
		\Vert J_l-J_0\Vert&\leq 4n\max\limits_j\Vert A_j\Vert\cdot \max\left\{ \Vert V_l-V_0\Vert,\Vert U_l-U_0\Vert \right\} \nonumber\\
		& \leq 8n \max\limits_j\Vert A_j \Vert\sum\limits_{k=0}^{l-1}(\max\left\{\|\overline{X}_k\|, \|\overline{Y}_k\| \right\} + \max\left\{\|\overline{E}_k\|, \|\overline{F}_k\| \right\}).\label{esallj_l-j_0}
	\end{align}
	That is, \eqref{esj_l-j_0} holds.  Note by \eqref{uk-ukover} (with $k=l$) that
	$\overline{U}_l-U_{l}=-U_{l}\overline{X}_{l}$ and $\overline{V}_l-V_{l}=-V_{l}\overline{Y}_{l}.$
	Then, using a similar arguments for \eqref{esallj_l-j_0}, one checks
	\begin{equation*}
		\begin{split}
			\| J_l-\overline{J}_l\|  \leq 8n \max\limits_j\Vert A_j\Vert\cdot \max\left\{ \| \overline{X}_l\|,\| \overline{Y}_l\| \right\}. \nonumber
		\end{split}
	\end{equation*}
	This shows \eqref{esj_l-oj_l}. It remains to prove \eqref{esj_l+1-j_l}. For this end, we assume that
	$ \max\left\{\Vert X_{l+1} \Vert,\Vert Y_{l+1} \Vert\right\} \leq 1$. Then, it follows from \eqref{ukvk} that
	\begin{equation}\nonumber
		\max\left\{ \|U_{l+1}\|,\; \|V_{l+1}\| \right\} \leq 2.
	\end{equation}
	Thus, using a similar arguments for \eqref{esallj_l-j_0},  we obtain by \eqref{uvuv}, \eqref{esuk+k-uk} and \eqref{esvk+k-vk} (with $k=l$) that
	\begin{equation}\nonumber
		\Vert J_{l+1}-J_l \Vert \leq 8n \max\limits_j\Vert A_j \Vert(\max\left\{ \|\overline{X}_k\|,\|\overline{Y}_k\| \right\}+\max\left\{ \|\overline{E}_k\|,\|\overline{F}_k\| \right\}).
	\end{equation}
	Therefore, \eqref{esj_l+1-j_l} holds and the proof is complete.
	
\end{proof}

Now we present the main theorem of this paper, which shows that the sequence $\{\bm{c}^k\}$ generated by \autoref{improved algorithm Ulm} converges  to the solution $\bm{c}^*$ with root-convergence rate at least cubically.
\begin{theorem}
	Suppose that $\bm{c}^*$ is a solution of the ISVP \eqref{iep} and  $J(\bm{c}^*)$ is nonsingular. Then there exist positive constants $ r $ and  $ \mu $ such that, for each $\bm{c}^0\in B(\bm{c}^*,r)$ and each $ B_0\in\mathbb{R}^{n\times n}$ satisfying
	\begin{equation}\label{I-BoJo}
		\Vert I-B_0J_0\Vert \leq \mu,
	\end{equation}
	the sequence $\{\bm{c}^k\}$ generated by \autoref{improved algorithm Ulm} is well-defined and converges to $\bm{c}^*$ with cubic root-convergence. Furthermore, we have
	\begin{equation*}
		\Vert \bm{c}^k-\bm{c}^* \Vert \leq Lr\left(\frac{1}{2}\right)^{3^k},
	\end{equation*}
	where $L$ is a positive constant.
\end{theorem}
\begin{proof}Let $r_0$ and $\eta_0$ be determined by \autoref{Q-Q<A-A}.
	According to \autoref{lemma j<k},  there exist positive numbers $ r_1 $ and $ \eta_1\geq 1 $ such that the following implication holds:
	\begin{equation}\label{imp}
		\bm{c}\in B(\bm{c}^*,r_1)\Longrightarrow  {\boldsymbol{f}'(\bm{c})}^{-1} \;{\rm  exists\; and}\;\Vert {\boldsymbol{f}'(\bm{c})}^{-1}  \Vert \leq \eta_1.
	\end{equation}
	Let  $\eta_2>1$ be defined as in \autoref{lemma impront}.
	For simplicity, we write
	\begin{equation}\nonumber
		\begin{aligned}
			\tau_1:=72\sqrt{m}\eta_2^2\Vert \Sigma^* \Vert,\quad \tau_2:=16n\eta_1(1+8{\eta}_2+5\eta_2^3)\max\limits_j\Vert A_j \Vert .\\
		\end{aligned}
	\end{equation}
	Set 	\begin{equation}\label{Ldefinition}L:=\left\{2,\ 2m\eta_0\max\limits_j\Vert A_j\Vert \right\}\end{equation}
	and
	\begin{align}\label{delta}
		r:=\min\biggl\{r_1,\; r_0, \;\frac{r_0}{\sqrt{m}\max\limits_j\Vert A_j \Vert },\;
		\frac{1}{3L\eta_2(1+2\tau_2+4\eta_1\tau_1+72\eta_2^2)^3}
		\biggr\}.
	\end{align}
	Let $0\leq \mu \leq r$. Obviously, $\mu<1$ as $r<1$ by definitions.
	
	Below we shall show that $L$, $r$ and $\mu$ are as desired. For this end,
	let $\bm{c}^0\in B(\bm{c}^*,r)$ and  $ B_0\in\mathbb{R}^{n\times n}$ such that
	\eqref{I-BoJo} holds.
	To complete the proof, it suffices  to prove the following inequalities hold for each $k\in\mathbb{N}$:	
	
	\begin{equation}\label{maxforck-c^*}
		\|\bm{c}^k-\bm{c}^*\|\leq Lr \left(\frac{1}{2}\right)^{3^k},
	\end{equation}
	\begin{equation}\label{maxforxy}
		\max\left\{ \|X_k\|,\;\|Y_k\|   \right\} \leq Lr \left(\frac{1}{2}\right)^{3^k},
	\end{equation}
	
	\begin{equation}\label{maxfori-bj}
		\Vert I-B_kJ_k\Vert\leq Lr \left(\frac{1}{2}\right)^{3^k},
	\end{equation}
	\begin{equation}\label{maxforoc-c^*}
		\Vert \overline{\bm{c}}^k-\bm{c}^* \Vert \leq Lr \left(\frac{1}{2}\right)^{2 \cdot 3^k}.
	\end{equation}
	Clearly, \eqref{maxforck-c^*} and \eqref{maxfori-bj} hold for $k=0$  by the facts that $\mu \leq r $ and $L\ge2$. To prove \eqref{maxforxy} for $k=0$, we first note by \eqref{defuv} (with $k=0$)
	and the orthogonality of $ U^*_0,\;V(\bm{c}^*) $ that
	\begin{equation*}
		\Vert {X}_0 \Vert_F =\Vert U^*_0{X}_0 \Vert_F=\Vert U_0-U^*_0 \Vert_F\quad{\rm and}	\quad\Vert {Y}_0 \Vert_F =\Vert V(\bm{c}^*){Y}_0 \Vert_F=\Vert V_0-V(\bm{c}^*) \Vert_F.
	\end{equation*}
	Then, it follows that
	\begin{equation}\label{U0}
		\Vert X_0\Vert\leq \Vert X_0\Vert_F\leq \sqrt{m} \Vert U_0-U^*_0\Vert\quad{\rm and}	\quad\Vert Y_0\Vert\leq \sqrt{n} \Vert V_0-V(\bm{c}^*) \Vert.
	\end{equation}
	Moreover, by \eqref{A} and the definition of $r$, one checks
	\begin{equation}\nonumber
		\|A(\bm{c}^0)-A(\bm{c}^*)\|_{F} \leq \sqrt{m}\|A(\bm{c}^0)-A(\bm{c}^*)\| \leq \sqrt{m} \max_{j}\|A_j\|\cdot \|\bm{c}^0-\bm{c}^*\| \leq r_0.
	\end{equation}
	Thus, \autoref{Q-Q<A-A} is applicable (with $A(\bm{c}^0)$, $A(\bm{c}^*)$, $U_0$, $U_0^*$,  $V_0$, $V(\bm{c}^*)$ in place of $A$, $\overline{A}$, $U$, $\overline{U}$,  $V$, $\overline{V}$) to conclude that
	\begin{equation*}
		\Vert U_0-U_0^* \Vert_{F} \leq \eta_0\|A(\bm{c}^0)-A(\bm{c}^*)\|_{F}\leq \sqrt{m}\eta_0 \max_{j}\|A_j\|\cdot \|\bm{c}^0-\bm{c}^*\|
	\end{equation*}
	and
	\begin{equation*}
		\Vert V_0-V(\bm{c}^*) \Vert_{F} \leq \eta_0\|A(\bm{c}^0)-A(\bm{c}^*)\|_{F}\leq \sqrt{m} \eta_0 \max_{j}\|A_j\|\cdot \|\bm{c}^0-\bm{c}^*\|.
	\end{equation*}
	Hence, thanks to \eqref{U0}, \eqref{Ldefinition} and the assumption  $\bm{c}^0\in B(\bm{c}^*,r)$, we obtain
	\begin{align}\nonumber
		\max\left\{\|X_0\|,\;\|Y_0\| \right\} \leq \frac{1}{2}Lr.
	\end{align}
	That is, \eqref{maxforxy} holds for $k=0$. Moreover, by \autoref{esalljk+bk}, we get
	\begin{equation}\label{Jmc*-d^m}
		\Vert J_0\bm{c}^*+\bm{b}^0\Vert \leq 2\sqrt{n}\Vert \Sigma^* \Vert \max\left\{{\|X_0\|}^2,{\|Y_0\|}^2 \right\}\leq \frac{1}{2}\sqrt{n}\Vert \Sigma^* \Vert L^2r ^2.
	\end{equation}Since $ J_0=f'(\bm{c}^0)$ and that $\bm{c}^0\in B(\bm{c}^*,r)\subset B(\bm{c}^*,r_1)$, we have by  \eqref{imp} that
	$J^{-1}_0$ exists and	\begin{equation}\label{J0-1}
		\Vert J^{-1}_0 \Vert \leq \eta_1.
	\end{equation}
	Then, thanks to  $ \eqref{I-BoJo} $ and the fact that $\mu\le1$, we have
	\begin{equation}\label{B0}
		\Vert B_0 \Vert \leq \Vert B_0J_0 \Vert\cdot \Vert J_0^{-1} \Vert \leq (1+\Vert I-B_0J_0 \Vert)\Vert J_0^{-1} \Vert \leq (1+\mu)\eta_1\leq 2\eta_1.\\
	\end{equation}
	Note that
	\begin{equation*}
		\overline{\bm{c}}^0-\bm{c}^*=\bm{c}^0-\bm{c}^*-B_0(J_0\bm{c}^0+\bm{b}^0)=(I-B_0J_0)(\bm{c}^0-\bm{c}^*)-B_0(J_0\bm{c}^*+\bm{b}^0).
	\end{equation*}
	Thus, using $\eqref{Jmc*-d^m} $, \eqref{B0} and the assumption $\bm{c}^0\in B(\bm{c}^*,r)$, one checks
	\begin{align}
		\Vert \overline{\bm{c}}^0-\bm{c}^* \Vert
		&\leq \Vert I-B_0J_0\Vert \cdot \Vert \bm{c}^0-\bm{c}^* \Vert + \Vert B_0 \Vert \cdot \Vert J_0\bm{c}^*+\bm{b}^0 \Vert\nonumber
		\leq(\frac{1}{4}+\eta_1\sqrt{n} \Vert \Sigma^* \Vert)L^2r^2.\label{c0-c^*}
	\end{align}
	Hence, \eqref{maxforoc-c^*} holds for $k=0$ as, by
	the definitions of $r$ and $\tau_1$,
	\begin{equation}\nonumber
		Lr \leq  \frac{1}{3\eta_2(1+2\tau_2+4\eta_1\tau_1+72\eta_2^2)^3} \leq \frac{1}{1+4\eta_1\sqrt{n} \Vert \Sigma^* \Vert} .
	\end{equation}

	Now we assume that \eqref{maxforck-c^*}--\eqref{maxforoc-c^*} hold for each $0\leq k\leq l$. Below we show that \eqref{maxforck-c^*}--\eqref{maxforoc-c^*}  hold for $k=l+1$. In fact, by \eqref{delta} and the inductive assumptions, one has
	\begin{equation}\nonumber
		\max\left\{{\|X_k\|},{\|Y_k\|} \right\}\leq Lr\left(\frac{1}{2}\right)^{3^{k}} \leq \frac12Lr \leq 1, \quad 0\leq k\leq l.
	\end{equation}
	Then, \autoref{lemma impront}(i) is applicable to conclude that
	\begin{equation*}
		\max\left\{ \|X_{k}-\overline{X}_{k}\|,\;\|Y_{k}-\overline{Y}_{k}\| \right\}\leq \eta_2({\|X_k\|}^2+{\|Y_k\|}^2+\Vert \overline{\bm{c}}^k-\bm{c}^* \Vert), \quad 0\leq k\leq l
	\end{equation*}
	and
	\begin{equation}\label{m4}
		\max\left\{ \|E_k\|,\;\|F_k\| \right\}\leq 3\eta_2({\|X_k\|}^2+{\|Y_k\|}^2+\Vert \overline{\bm{c}}^k-\bm{c}^* \Vert), \quad 0\leq k\leq l.
	\end{equation}
	Thus, by the inductive assumptions of \eqref{maxforxy} and \eqref{maxforoc-c^*}, we have
	\begin{equation}\begin{array}{lll}\label{wanFk}
			\max\left\{ \Vert\overline{X}_{k}\Vert,\;\Vert\overline{Y}_{k}\Vert \right\}
			&\leq\max\left\{ \|X_{k}\|+\|X_{k}-\overline{X}_{k}\|,\|Y_{k}\|+\|Y_{k}-\overline{Y}_{k}\|\right\}\\
			&\leq Lr(1+\eta_2(\frac{1}{2})^{3^{k}} +2\eta_2 Lr{(\frac{1}{2})}^{3^{k}}){(\frac{1}{2})}^{3^{k}}\\
			&\leq (1+2\eta_2)Lr\left(\frac{1}{2}\right)^{3^{k}}, \quad  0\leq k\leq l,
	\end{array}\end{equation}
	where the last inequality holds because of the fact that $Lr\leq1$.
	Note by the definition of $r$ that
	\begin{equation}\label{Lr13}
		r\leq\frac{1}{3L\eta_2(1+2\tau_2+4\eta_1\tau_1+72\eta_2^2)^3}\leq\frac{1}{3L\eta_2}\le\frac{1}{3L}.\nonumber
	\end{equation}
	Thus, by the inductive assumptions of \eqref{maxforxy} and \eqref{maxforoc-c^*}, we derive from \eqref{m4} that
	\begin{equation}\label{Gk}
		\max\left\{{\|E_k\|},{\|F_k\|} \right\} \leq   3\eta_2 Lr(1+2Lr)\left(\frac{1}{2}\right)^{2\cdot 3^{k}}<1,\quad 0\leq k\leq l.
	\end{equation}
	Hence, applying \autoref{lemma impront}(ii), one sees that  \eqref{ek-overek} and \eqref{xk+1yk+1} hold for each $0\leq k\leq l$.
	Therefore, using  \eqref{Gk} and the inductive assumption of \eqref{maxforck-c^*}, we deduce
	\begin{align}
		\max\left\{ \Vert\overline{E}_{k}\Vert,\Vert\overline{F}_{k}\Vert \right\}\nonumber
		&\leq\max\left\{ \|E_{k}-\overline{E}_{k}\|,\|F_{k}-\overline{F}_{k}\|\right\}+\max\left\{ \|E_{k}\|,\|F_{k}\|\right\}\nonumber\\
		&\leq\eta_2 \left[
		\left(\frac{1}{2}\right)^{3^{k}}+3(1+2Lr)+18\eta_2^2Lr(1+2Lr)^2\left(\frac{1}{2}\right)^{2 \cdot 3^{k}}\right]
		Lr\left(\frac{1}{2}\right)^{ 2 \cdot 3^{k}}\nonumber\\
		&\leq \eta_2(6+5\eta_2^2)Lr\left(\frac{1}{2}\right)^{3^{k}},\quad0\leq k\leq l-1.\label{Fm-1}
	\end{align}
	where the last inequality holds because of the fact that $Lr \leq \frac{1}{3}$ (see \eqref{Lr13}). Thanks to \eqref{Gk}, \eqref{maxforxy} and  \autoref{esj_x-j_y}, one sees
	\begin{equation}\nonumber
		\Vert J_l-J_0\Vert
		\leq 8n \max\limits_j\Vert A_j \Vert\sum\limits_{k=0}^{l-1}(\max\left\{\|\overline{X}_k\|, \|\overline{Y}_k\| \right\} + \max\left\{\|\overline{E}_k\|, \|\overline{F}_k\| \right\})
	\end{equation}
	and
	\begin{equation}\label{j_l=oj_l}
		\Vert J_l- \overline{J}_l \Vert\leq 8n\max\limits_j\Vert A_j \Vert\cdot\max\left\{ \Vert \overline{X}_l\Vert,\Vert \overline{Y}_l \Vert \right\}.
	\end{equation}
	Thus, by \eqref{J0-1}, \eqref{wanFk}, \eqref{Fm-1} and the  definition of $r$, we deduce
	\begin{equation*}
		\Vert J_0^{-1}\Vert\cdot \Vert J_l-J_0\Vert\leq  8n\eta_1(1+8\eta_2+5\eta_2^3)\max\limits_j\Vert A_j\Vert Lr \leq \frac{1}{2}\tau_2Lr \leq \frac12.
	\end{equation*}
	Thus, \autoref{lemma AB-1} is applicable to conclude that $J_l^{-1}$ exists and
	\begin{equation}\label{Jlinv}
		\Vert J_l^{-1} \Vert \leq \frac{\Vert J_0^{-1} \Vert}{1-\Vert J_0^{-1} \Vert\cdot\Vert J_l-J_0 \Vert}\leq 2\eta_1.
	\end{equation}
	Hence, using the inductive assumption of \eqref{maxfori-bj} and noting the fact that $Lr\le2$, we have
	\begin{equation}\label{bm}
		\Vert B_l \Vert \leq \Vert B_lJ_l \Vert\cdot \Vert J_l^{-1} \Vert \leq (1+\Vert I-B_lJ_l \Vert)\Vert J_l^{-1} \Vert \leq 4\eta_1,\\
	\end{equation}
	which together with \eqref{j_l=oj_l} gives
	\begin{align*}
		\Vert I-B_l \overline{J}_l \Vert &\leq \Vert I-B_lJ_l \Vert +\Vert B_l \Vert \cdot \Vert J_l- \overline{J}_l \Vert \\
		&\leq \Vert I-B_lJ_l \Vert +32n\eta_1\max\limits_j\Vert A_j \Vert\max\left\{ \Vert \overline{X}_l\Vert,\Vert \overline{Y}_l \Vert \right\}.
	\end{align*}
	Substituting  \eqref{maxfori-bj} (with $k=l$) and  \eqref{wanFk} into the above inequality, one has
	\begin{align}
		\Vert I-B_l \overline{J}_l \Vert
		&\leq (1+32n\eta_1(1+2\eta_2)\max\limits_j\Vert A_j \Vert)Lr\left(\frac{1}{2}\right)^{3^l} \nonumber\\
		&\leq (1+2\tau_2)Lr\left(\frac{1}{2}\right)^{3^l}.\label{i-bmj2m}
	\end{align}
	where the last inequality holds because of the fact that  $\tau_2=16n\eta_1(1+8\eta_2+5\eta_2^3)\max\limits_j\Vert A_j \Vert \ge 16n\eta_1(1+2\eta_2)\max\limits_j\Vert A_j \Vert$.
	On the other hand, by \autoref{esalljk+bk}, \eqref{Gk} and the definition of $\tau_1$, we get
	\begin{align}
		\Vert \overline{J}_l\bm{c}^*+\overline{\bm{b}}^l\Vert
		&	\leq 2\sqrt{n} \|\Sigma^* \| \max\left\{{\|E_l\|}^2,{\|F_l\|}^2 \right\}\nonumber\\
		&	\leq 18\sqrt{n}\eta_2^2 \|\Sigma^*\| (1+2Lr)^2L^2r^2\left(\frac{1}{2}\right)^{4\cdot3^l}\nonumber\\
		&	\leq \tau_1 L^2r^2 \left(\frac{1}{2}\right)^{4\cdot3^l}.\label{esojs+bs}
	\end{align}
	Then, noting that
	\begin{equation*}
		\bm{c}^{l+1}-\bm{c}^*=\overline{\bm{c}}^l-\bm{c}^*-B_l\boldsymbol{\rho}^l
		=(I-B_l \overline{J}_l)(\overline{\bm{c}}^l-\bm{c}^*)-B_l( \overline{J}_l\bm{c}^*+\overline{\bm{b}}^l),
	\end{equation*}
	we have  by \eqref{maxforoc-c^*} (with $k=l$) and \eqref{bm}--\eqref{esojs+bs} that
	\begin{align}
		\Vert \bm{c}^{l+1}-\bm{c}^* \Vert
		&\leq (1+2\tau_2)L^2r^2\left(\dfrac{1}{2}\right)^{3^{l+1}} +4\eta_1\tau_1 L^2r^2 \left(\dfrac{1}{2}\right)^{4\cdot3^l}\nonumber\\
		&\leq (1+2\tau_2+4\eta_1\tau_1)L^2r^2\left(\dfrac{1}{2}\right)^{3^{l+1}}.\label{cm+1-c*}
	\end{align}
	Hence, \eqref{maxforck-c^*} is seen to  hold for $k=l+1$ by definition of $r$. Therefore, in view of
	\eqref{xk+1yk+1} (with $k=l$), \eqref{Gk} and \eqref{cm+1-c*}, we get
	\begin{align}
		\max\left\{ \|X_{l+1}\| ,\|Y_{l+1}\|  \right\}
		&\leq 3\eta_2 ({\|E_l\|}^2+{\|F_l\|}^2 + \|\bm{c}^{l+1} - \bm{c}^*\|)\nonumber \\
		&\leq 3\eta_2(18\eta_2^2(1+2Lr)^2+(1+2\tau_2+4\eta_1\tau_1))L^2r^2\left(\dfrac{1}{2}\right)^{3^{l+1}}.\nonumber	
	\end{align}
	This shows \eqref{maxforxy} for $k=l+1$ (noting $Lr<\frac13$ and  $3\eta_2\left (1+2\tau_2+4\eta_1\tau_1+72\eta_2^2\right )Lr<1$ by  the definition of $r$). Noting that \eqref{ek-overek} holds for $k=l$, by \eqref{maxforck-c^*}(with $k=l+1$) and \eqref{Gk}, we get that
	\begin{align}
		\max\left\{ \Vert\overline{E}_{l}\Vert,\Vert\overline{F}_{l}\Vert \right\}\nonumber
		&\leq\max\left\{ \|E_{l}-\overline{E}_{l}\|,\|F_{l}-\overline{F}_{l}\|\right\}+\max\left\{ \|E_{l}\|,\|F_{l}\|\right\}\nonumber\\
&\leq\eta_2 \left[
		\left(\frac{1}{2}\right)^{3^{l}}+3(1+2Lr)+18\eta_2^2Lr(1+2Lr)^2\left(\frac{1}{2}\right)^{2 \cdot 3^{l}}\right]
		Lr\left(\frac{1}{2}\right)^{ 2 \cdot 3^{l}}\nonumber\\
&\leq \eta_2(6+5\eta_2^2)Lr\left(\frac{1}{2}\right)^{3^{l}}.\label{esel+1}
	\end{align}
	Furthermore, one sees $\max\left\{ \|X_{l+1}\| ,\|Y_{l+1}\|  \right\} \leq 1$.  Thus, \autoref{esj_x-j_y} is applicable and so,
	\begin{equation*}
		\Vert J_{l+1}-J_l \Vert \leq 8n \max\limits_j\Vert A_j \Vert(\max\left\{ \|\overline{X}_l\|,\|\overline{Y}_l\| \right\}+\max\left\{ \|\overline{E}_l\|,\|\overline{F}_l\| \right\}),
	\end{equation*}which together with
	\eqref{wanFk} (with $k=l$) and \eqref{esel+1} yields
	\begin{equation}\label{Jm+1-Jm}
		\Vert J_{l+1}-J_l \Vert \leq 8n(1+8\eta_2+5\eta_2^3)\max\limits_j\Vert A_j \Vert Lr\left(\frac{1}{2}\right)^{3^l} .
	\end{equation}
	Noting that $B_{l+1}=B_l+B_l(2I-J_{l+1}B_l)(I-J_{l+1}B_l)$, we derive
	\begin{equation*}
		I-B_{l+1}J_{l+1}=(I-B_lJ_{l+1})^3=\left (I-B_lJ_l+B_l(J_l-J_{l+1})\right )^3.
	\end{equation*}
	Thus, by \eqref{maxfori-bj} (with $k=l$), \eqref{bm} and \eqref{Jm+1-Jm} , we obtain
	\begin{align}
		\Vert I-B_{l+1}J_{l+1} \Vert &\leq (1+32n\eta_1(1+8\eta_2+5\eta_2^3)\max\limits_j\Vert A_j \Vert)^3L^3r^3\left(\frac{1}{2}\right)^{3^{l+1}} \nonumber\\
		&=(1+2\tau_2)^3L^3r^3\left(\frac{1}{2}\right)^{3^{l+1}},\nonumber
	\end{align}
	where the equality holds because of the definition of $\tau_2$.  Therefore, \eqref{maxfori-bj} holds for $ k=l+1$ as
	\begin{equation}
		Lr \leq \frac{1}{3L\eta_2(1+2\tau_2+4\eta_1\tau_1+72\eta_2^2)^3}\leq \frac{1}{(1+2\tau_2)^3} \leq 1\nonumber
	\end{equation}
	(noting the definition of $r$). In order to prove \eqref{maxforoc-c^*}  for $ k=l+1 $, we first note by  \eqref{Jlinv}, \eqref{Jm+1-Jm} and the definitions of $\tau_2$, $r$ that
	\begin{equation*}
		\Vert J_{l}^{-1}\Vert\cdot\Vert J_{l+1}-J_l \Vert \leq 16n\eta_1(1+8\eta_2+5\eta_2^3)\max\limits_j\Vert A_j \Vert Lr\left(\frac{1}{2}\right)^{3^l}=\tau_2 Lr\left(\frac{1}{2}\right)^{3^l}\le\frac12,
	\end{equation*}
	and so \autoref{lemma AB-1} is applicable to conclude that   $J_{l+1}^{-1}$ exists and
	$\|J_{l+1}^{-1}\|\le 4\eta_1$. Then, using a similar arguments for $k=0$, one can easily check that \eqref{maxforoc-c^*} holds for $ k=l+1 $.
	Thus, we have established the validity of  \eqref{maxforck-c^*}--\eqref{maxforoc-c^*}  for each $k\ge0$.
	
	It remains to prove  the cubic root-convergence rate of the proposed algorithm. In fact, based on  \eqref{maxforck-c^*}, we can directly derive the root-convergence rate of the sequence $\{\bm{c}^k\}$ as follows. \\
	$\mathrm{1}$.  If $p=1$,
	\begin{equation*}
		R_1\{\bm{c}^k\}=	\limsup\limits_{k\to\infty}\Vert\bm{c}^k-\bm{c}^*\Vert^{\frac{1}{k}}\leq\limsup\limits_{k\to\infty}Lr\left( \frac{1}{2}\right)^{\frac{3^k}{k}}=0.\\
	\end{equation*}
	$2$. If $1<p<3$,
	\begin{equation*}
		R_p\{\bm{c}^k\}=	\limsup\limits_{k\to\infty}\Vert\bm{c}^k-\bm{c}^*\Vert^{\frac{1}{p^k}}\leq\limsup\limits_{k\to\infty}Lr\left( \frac{1}{2} \right)^{\frac{3^k}{p^k}}=0.\\
	\end{equation*}
	$3$. If $p \geq 3$,
	\begin{equation*}
		R_p\{\bm{c}^k\}=	\limsup\limits_{k\to\infty}\Vert\bm{c}^k-\bm{c}^*\Vert^{\frac{1}{p^k}}\leq\limsup\limits_{k\to\infty}(Lr)^{\frac{1}{p^k}}\left( \frac{1}{2}\right)^{\frac{3^k}{p^k}}= 1.\\
	\end{equation*}
	Then, according to Definition \ref{defroot}, one sees that $O_R\{\bm{c}^*\}\geq 3$ and so the proof is complete.
\end{proof}

\section{Numerical tests}\label{section4}
\noindent In this section, we implement \autoref{improved algorithm Ulm} proposed in Section s on an example which has  been presented in \cite{ulm dist}. For comparison, we also test in this example some other  existing algorithms for the ISVP \eqref{iep}, including \autoref{algorithm Ulm} and the two-step inexact Newton method in \cite{2 inexact newton} (for short, Algorithm TIN).
It should be noted that all numerical tests are  performed in MATLAB R2017b on a laptop equipped with an AMD Ryzen 5 4600H CPU (3.0 GHz). Furthermore, all linear systems are solved by the QMR function provided in MATLAB. The iterations of all algorithms are terminated when $k\geq50$ or the residual $$d_k:=	{\Vert U_k^TA(\bm{c}^k)V_k-\Sigma^*\Vert }_F$$ satisfies
$ d_k  \leq {10}^{-10}$.

\begin{example} 
We first generate the matrices $\left\{A_i\right\}_{i=0}^n\subset\mathbb{R}^{m\times n}$ using the \textnormal{rand} function in MATLAB. For demonstration purpose, we  also generate $\bm{c}^*$ by the \textbf{rand} function and compute the singular values of $A(\bm{c}^*)$ as the given singular values $\left\{\sigma_i^*\right\}_{i=1}^n$.
\end{example}

As in \cite{ulm dist}, we focus on the following three cases: $(a)\ m=100,\ n=60;\ (b)\ m=300, \ n=120;\ (c)\ m=600,\ n=300.$ Since \autoref{improved algorithm Ulm} converges locally, the initial guess $\bm{c}^0$ is generated by perturbing each entry of $\bm{c}^*$ with a uniformly distributed random value drawn from the interval $[ -\max\limits_{j} |c^{*}_{j}| \cdot \beta,\ \max\limits_{j} |c^{*}_{j}| \cdot \beta ]$.
Moreover, the matrix $B_0$ in \autoref{improved algorithm Ulm} is chosen to satisfy \eqref{I-BoJo}.
 Tables \ref{casea}--\ref{casec} present the values of $d_k$  respectively for  Cases ($a$)--($c$) generated by  \autoref{improved algorithm Ulm} with different values of $\beta$ and $\mu$.
From Tables \ref{casea}--\ref{casec}, we observe that it takes $2$ to $4$ steps for the convergence of \autoref{improved algorithm Ulm}.
Moreover, the smaller the values of $\mu$ and $\beta$, the faster the algorithm converges. This implies that an appropriate $B_0$ and initial point $\bm{c}^0$ are important for less iterations. Specifically, the convergence performance of \autoref{improved algorithm Ulm} with $\mu \leq 0.05$ is comparable to that of $\mu = 0.$


\begin{table}[!htbp]
	\centering
	\caption{the values of $d_k$ generated by \autoref{improved algorithm Ulm} for Case $(a)$ }
	\label{casea}
	\resizebox{\textwidth}{!}{ 
		\begin{tabular}{ccccccc}
			\hline
			{$\beta$} & {$k$} &{$\mu=0$} & {$\mu=0.001$} & {$\mu=0.005$} & {$\mu=0.01$} &{$\mu=0.05$}\\
			\hline
			$1e-3$ & $0$ & $7.38e-1$ & $7.38e-1$   & $7.38e-1$     & $7.38e-1$   & $7.38e-1$\\
			& $1$ & $3.56e-3$ & $2.73e-3$   & $5.69e-3$     & $9.65e-2$   & $6.33e-2$\\
			& $2$ & $7.89e-7$ & $6.59e-7$   & $7.62e-7$     & $3.32e-6$   & $2.59e-6$ \\
			& $3$ & $1.87e-14$& $2.15e-14$  & $2.33e-14$    & $2.74e-13$  & $1.58e-13$ \\
			& $\vdots$ & $\vdots$  & $\vdots$    &$\vdots$             &  $\vdots$		    &     $\vdots$ \\
			
			$1e-4$ & $0$ & $8.03e-2$ & $8.03e-2$   & $8.03e-2$     & $8.03e-2$   & $8.03e-2$ \\
			& $1$ & $3.46e-7$ & $9.31e-6$   & $1.52e-6$     & $1.65e-3$   & $2.02e-3$ \\
			& $2$ & $1.62e-14$& $1.89e-13$  & $4.38e-13$    & $4.32e-7$   & $5.59e-7$  \\
			& $3$ &   $\vdots$      &$\vdots$           &  $\vdots$           & $1.74e-14$  & $1.06e-14$\\
			& $\vdots$ &$\vdots$  &$\vdots$   & $\vdots$   &  $\vdots$   &$\vdots$\\
			\hline
		\end{tabular}
	}
\end{table}

\begin{table}[!htbp]
	\centering
	\caption{the values of $d_k$ generated by \autoref{improved algorithm Ulm} for Case $(b)$}
	\label{caseb}
	\resizebox{\textwidth}{!}{
		\begin{tabular}{ccccccc}
			\hline
			{$\beta$} & {$k$} &{$\mu=0$} & {$\mu=0.001$} & {$\mu=0.005$} & {$\mu=0.01$} &{$\mu=0.05$}\\
			\hline
			$1e-3$ & $0$  & $3.61e+0$   & $3.61e+0$     & $3.61e+0$   & $3.61e+0$  & $3.61e+0$ \\
			& $1$  & $8.21e-3$   & $9.22e-3$     & $3.29e-1$   & $5.12e-1$  & $3.25e-1$ \\
			& $2$  & $4.23e-6$   & $6.26e-6$     & $7.21e-3$   & $8.15e-3$  & $7.11e-3$ \\
			& $3$  & $7.10e-13$  & $6.31e-13$    & $3.22e-6$   & $5.12e-6$  & $2.24e-6$ \\
			& $4$  &  $\vdots$   &   $\vdots$    & $1.98e-13$  & $4.57e-13$ & $3.21e-13$ \\
			& $\vdots$ &  $\vdots$  & $\vdots$      & $\vdots$    &$\vdots$    &  $\vdots$	\\
			
			$1e-4$ & $0$  & $8.69e-1$   & $8.69e-1$    & $8.69e-1$   & $8.69e-1$   & $8.69e-1$\\
			& $1$  & $6.35e-3$   & $5.51e-3$    & $3.22e-3$   & $9.18e-3$   & $8.33e-3$ \\
			& $2$  & $7.03e-7$   & $6.22e-7$    & $6.83e-7$   & $5.11e-6$   & $3.59e-6$ \\
			& $3$  & $1.32e-14$  & $2.62e-14$   & $1.33e-14$  & $2.26e-13$  & $1.14e-13$\\
			& $\vdots$  &	$\vdots$   &  $\vdots$    &    $\vdots$  &  $\vdots$   &       $\vdots$ \\
			\hline
		\end{tabular}
	}
\end{table}
\begin{table}[!htbp]
	\centering
	\caption{the values of $d_k$ generated by \autoref{improved algorithm Ulm} for Case $(c)$}
	\label{casec}
	\resizebox{\textwidth}{!}{
		\begin{tabular}{ccccccc}
			\hline
			{$\beta$} & {$k$} &{$\mu=0$} & {$\mu=0.001$} & {$\mu=0.005$} & {$\mu=0.01$} &{$\mu=0.05$}\\
			\hline
			$1e-4$ & $0$  & $5.54e-1$  & $5.54e-1$    & $5.54e-1$  &$5.54e-1$    &$5.54e-1$ 	\\
			& $1$  & $7.32e-3$  & $8.32e-3$    & $8.03e-3$	&$7.21e-3$    &$8.97e-3$ \\
			& $2$  & $8.32e-7$  & $7.11e-7$    & $6.99e-7$ 	&$3.54e-6$    &$6.07e-6$\\
			& $3$  & $1.12e-14$ & $2.69e-14$   & $2.68e-14$ &$7.66e-13$   &$7.24e-13$  \\
			& $\vdots$  &  $\vdots$  &  $\vdots$    &  $\vdots$	&    $\vdots$  &  	$\vdots$\\
			
			$1e-5$ & $0$  & $7.22e-2$  & $7.22e-2$    &$7.22e-2$   &$7.22e-2$     &$7.22e-2$ \\
			& $1$  & $6.98e-5$  & $8.51e-5$    & $3.63e-5$ 	&$2.64e-3$     &$2.96e-3$\\
			& $2$  & $5.61e-14$ & $7.52e-14$   & $1.91e-13$ &$8.48e-7$     &$6.54e-7$ \\
			& $3$  & $\vdots$   &$\vdots$      & $\vdots$ 	&$1.98e-14$	 &$4.02e-14$ \\
			&$\vdots$ &$\vdots$  &$\vdots$   &  $\vdots$  & $\vdots$    &  $\vdots$ \\
			\hline
		\end{tabular}
	}
\end{table}
For comparison, we also conducted experiments for \autoref{algorithm Ulm}, \autoref{improved algorithm Ulm} and  Algorithm TIN with ten initial points for each problem size, where $\beta$ is set to be $1e-3$, $1e-4$ and $1e-5$ respectively for Cases (a)--(c). Particularly, we set $\mu=0$ in \autoref{improved algorithm Ulm} for simplicity.
Table \ref{tab:residualsc} presents the numerical experimental results of all algorithms with the initial points ensuring convergence,
where $\textbf{Ite.}$ and  $\textbf{Tim.}$ denote the averaged outer iteration number and  CPU time (in seconds), respectively. While Table \ref{table:stability} presents one of the numerical results of  Case (c), where $\textbf{condJ}$ denotes the condition number of $J_k$.
As shown in Table \ref{tab:residualsc}, the proposed \autoref{improved algorithm Ulm} demonstrates superior computational efficiency, requiring less CPU time to achieve convergence compared to other algorithms.
This performance advantage of \autoref{improved algorithm Ulm} may stem from the elimination‌‌
of the Cayley transformation (and so the solution of $2(m+n)$ linear systems) at each iteration.
Moreover, Table \ref{table:stability} reveals that when the condition numbers of approximate Jacobian matrices become large, both \autoref{algorithm Ulm} and  \autoref{improved algorithm Ulm} demonstrate significantly better performance than Algorithm TIN.  This may be due to the fact that, compared to Algorithm TIN, \autoref{algorithm Ulm} and \autoref{improved algorithm Ulm} do not require solving the approximate Jacobian equation at each iteration.
Consequently, under high condition number of $J_k$, \autoref{algorithm Ulm} and \autoref{improved algorithm Ulm} exhibit  greater stability than Algorithm TIN.

\begin{table}[!htbp]
	\centering
	\caption{The averaged outer iteration number and CPU time }
	\label{tab:residualsc}
	\begin{tabular}{ccccccc}
		\toprule
		\multirow{2}{*}{Case} & \multicolumn{2}{c}{\autoref{algorithm Ulm}} & \multicolumn{2}{c}{\autoref{improved algorithm Ulm}} & \multicolumn{2}{c}{Algorithm TIN} \\
		\cmidrule(lr){2-3} \cmidrule(lr){4-5} \cmidrule(lr){6-7}
		& $\textbf{Tim.}$ & $\textbf{Ite.}$  & $\textbf{Tim.}$ & $\textbf{Ite.}$ & $\textbf{Tim.}$ & $\textbf{Ite.}$  \\
		\midrule
		$(a)$  & $0.47$     & $3.20$     & $0.36$   & $3.20$     & $0.48$  & $3.20$   \\
		\addlinespace[3pt] 
		$(b)$  & $8.51$     & $3.10$     & $7.52$   & $3.10$     & $8.72$  & $3.10$   \\
		\addlinespace[3pt] 
		$(c)$  & $266.58$   & $2.50$     & $241.02$ & $2.50$     & $271.71$& $2.50$    \\
		\bottomrule
	\end{tabular}
\end{table}

\begin{table}[!htbp]
    \centering
    \caption{the values of $d_k$ generated by all algorithms for Case (c)}
    \label{table:stability}
    \begin{tabular}{ccccccc}
        \toprule
        \multirow{2}{*}{$k$} & \multicolumn{2}{c}{\autoref{algorithm Ulm}} & \multicolumn{2}{c}{\autoref{improved algorithm Ulm}} & \multicolumn{2}{c}{Algorithm TIN} \\
        \cmidrule(lr){2-3} \cmidrule(lr){4-5} \cmidrule(lr){6-7}
        & {$d_k$} & {$\textbf{condJ}$}  & {$d_k$} & {$\textbf{condJ}$}  & {$d_k$} & {$\textbf{condJ}$}  \\
        \midrule
        $0$  & $8.59e-2$   & $8.11e+6$     & $8.59e-2$   & $8.11e+6$      & $8.59e-2$  & $8.11e+6$    \\
        $1$  & $2.11e-7$   & $2.21e+7$     & $3.29e-7$   & $2.21e+7$      & $7.51e-2$  & $2.21e+7$    \\
        $2$  & $3.71e-11$  & $\vdots$      & $2.59e-11$  & $\vdots$       & $4.91e-3$  & $2.29e+7$    \\
        $3$  & $\vdots$    & $\vdots$      & $\vdots$    & $\vdots$       & $3.65e-3$  & $2.37e+7$    \\
        $\vdots$  & $\vdots$   & $\vdots$     & $\vdots$    & $\vdots$       & $\vdots$   & $\vdots$     \\
        $50$ & $\vdots$    & $\vdots$      & $\vdots$    & $\vdots$       & $9.18e-3$  & $2.42e+7$    \\
        \bottomrule
    \end{tabular}
\end{table}

\section*{Declarations}
\begin{competing}
The authors declare no competing interests.
\end{competing}

\begin{ethics}
Not applicable.
\end{ethics}

\begin{consent}
All authors agreed to publish.
\end{consent}
\begin{constforpar}
All authors participated.
\end{constforpar}

\begin{data}
No datasets were generated or analysed during the current study.
\end{data}

\begin{materials}
Not applicable.
\end{materials}

\begin{code}
Not applicable.
\end{code}

\begin{authorcon}
All authors contributed to writing, review and editing. All authors
approved the final manuscript.
\end{authorcon}

\end{document}